\documentclass[letterpaper,11pt,twoside,reqno]{amsart}
\usepackage[top=1.5in,bottom=1.5in,left=1.4in,right=1.4in]{geometry}
\usepackage[pagebackref,colorlinks=true,urlcolor=blue,linkcolor=blue,citecolor=blue]{hyperref}
\usepackage{times}
\usepackage{amsmath}
\usepackage{amssymb}
\usepackage{amsfonts}
\usepackage{amscd}
\usepackage{amsthm}
\usepackage{amsxtra}
\usepackage{stmaryrd}

\usepackage[all]{xy}  \CompileMatrices
\usepackage{mathrsfs}
\usepackage{nicefrac}
\usepackage{graphicx}
\usepackage{color}
\usepackage{enumerate}
\usepackage{wrapfig}
\usepackage{url}

\usepackage{booktabs}
\usepackage{algorithm}

\theoremstyle{plain}
\newtheorem{theorem}{Theorem}
\newtheorem{proposition}[theorem]{Proposition}
\newtheorem*{proposition*}{Proposition}
\newtheorem{corollary}[theorem]{Corollary}
\newtheorem*{corollary*}{Corollary}
\newtheorem{lemma}[theorem]{Lemma}

\newtheorem*{theorem*}{Theorem}
\newtheorem*{lemma*}{Lemma}
\newtheorem*{conjecture*}{Conjecture}

\newtheorem*{question*}{Question}

\newtheorem*{problem*}{Problem}

\theoremstyle{definition}

\newtheorem*{exercise*}{Exercise}

\theoremstyle{remark}
\newtheorem{remark}[theorem]{Remark}
\newtheorem*{remark*}{Remark}
\newtheorem{remsTh}[theorem]{Remarks}

\newcommand{\subclass}[1]{}

\newcommand{\enumTi}[1]{\renewcommand{\theenumi}{#1}}

\newcommand{\alphenumi}{\enumTi{\alph{enumi}}}

\newcommand{\romenumi}{\enumTi{\roman{enumi}}}

\newcommand{\aand}{\mathrel\&}

\newcommand{\gecmt}[1]{\mathrel{\mathop\ge\limits_{#1}}}

\newcommand{\lt}{\left}
\newcommand{\rt}{\right}

\newcommand{\abs}[1]{{\lt\lvert{#1}\rt\rvert}}

\newcommand{\sabs}[1]{{\lvert{#1}\rvert}}

\newcommand{\widebar}[1]{\overline{#1}}

\newcommand{\nfrac}[2]{{\nicefrac{#1}{#2}}}

\newcommand{\RR}{\mathbb{R}}

\DeclareMathOperator{\Prb}{\mathbf{P}}
\DeclareMathOperator{\Exp}{\mathbf{E}}

\DeclareMathOperator{\IndicatorOp}{\mathbf{I}}
\newcommand{\Ind}{\IndicatorOp}

\newcommand{\eps}{\varepsilon}

\newlength{\algotabbingwidth}
\def\lor{\vee}
\def\land{\wedge}
\def\TRUE{{{\textsc{\small True}}}}
\def\FALSE{{{\textsc{\small False}}}}

\DeclareMathOperator{\Bin}{Bin}
\newcommand{\Hist}{\mathscr H}
\newcommand{\Histt}{\Hist\mspace{-2mu}(t)}

\DeclareMathOperator{\polylog}{polylog}

\newcommand{\aclause}{\texttt{C}}
\newcommand{\aliteral}{\texttt{L}}
\newcommand{\anatom}{\aliteral}
\newcommand{\avar}{\texttt{x}}
\newcommand{\anothervar}{\texttt{y}}
\newcommand{\anI}{\texttt{I}} 
\newcommand{\aJ}{\texttt{J}}  
\newcommand{\anS}{\texttt{S}}  

\begin{document}

\title[Random 3-iSAT]{An algorithm for random signed 3-SAT with Intervals}


\author{Kathrin Ballerstein}%
\address{Kathrin Ballerstein: Institute for Operations Research\\
Department of Mathematics\\
ETH Zurich\\
R\"amistrasse 101\\
8092 Zurich\\
Switzerland}%
\email{kathrin.ballerstein@ifor.math.ethz.ch}%

\author{Dirk Oliver Theis}
\address{Dirk Oliver Theis:
  University of Tartu\\
  Institute of Computer Science\\
  J.~Liivi 2\\
  50409~Tartu\\
  Estonia}%
\curraddr{}%
\email{dirk.oliver.theis@ut.ee \tiny\href{http://dirkolivertheis.blogspot.com}{http://dirkolivertheis.blogspot.com}}%

\subjclass[2000]{Primary XXXXX}

\date{Wed Aug 14 19:25:32 EEST 2013}




\begin{abstract}
  Interval-$k$-SAT ($k$-iSAT) is a generalization of classical $k$-SAT where the variables can take values
  in $[0,1]$ (instead of $\{0,1\}$) and the literals are of the form $x\in I$, for intervals
  $I\subset[0,1]$.  It falls within the class of signed satisfyability problems.

  We propose an algorithm for $3$-iSAT, and analyze it on uniformly random formulas.  The algorithm follows
  the Unit Clause paradigm, enhanced by a (very limited) backtracking option.  Using Wormald's ODE method,
  we prove that, if $m/n \le 2.3$, with high probability, our algorithm succeeds in finding an assignment
  of values to the variables satisfying the formula.
\end{abstract}


\maketitle



\section{Introduction}

Let $M$ be a (usually finite) set, $\mathcal S$ a set of subsets of $M$, and $X$ a set of variables. A \textit{(signed) literal} is the pair $(\avar,\anS)\in X\times
\mathcal S$, which we will denote as $\avar \in \anS$, and for a positive integer $k$, a \textit{$k$-clause} (or simply clause) is the disjunction ($\lor$) of at
most $k$ literals. The conjunction ($\land$) of finitely many $k$-clauses is called the \textit{signed $k$ conjunctive normal form ($k$-CNF)}. In this setting the
central question is the \textit{signed $k$-satisfiability problem}, or
\textit{signed $k$-SAT}, which asks for a satisfying \textit{interpretation}, that is, an assignment of values to
the variables such that in each clause there is at least one literal $(\avar,\anS)$ for which $\avar$ takes a value in $\anS$.

This setting includes as a special case the classical satisfiability (SAT) problem.  There, one chooses for $M$ the 2-element set $\{\TRUE,\FALSE\}$ and $\mathcal
S=\{\{\TRUE\}, \{\FALSE\}\}$. In case $M$ is an ordered set (a chain) and the set $\mathcal S$ is the set of all intervals in $M$, we speak of \textit{Interval
  SAT}, or \textit{iSAT}.  In our contribution, we set $M:=[0,1]$, because this includes all iSAT settings with finite $M$. In particular, we consider formulas of
the type
\begin{align*}
  \bigwedge_{i=1}^t \ \bigvee_{j\in \mathcal J_i} \avar_j \in \texttt{I}_j^i,
\end{align*}
where, for all $i=1,\ldots,t$, $\mathcal J_i$, with $|\mathcal J_i|\leq 3$, is an index set of variables in $X$, and $\texttt{I}_j^i \subseteq [0,1]$ are intervals
for all $i$ and $j$. Then, an interpretation of a clause $i$ is satisfying if there is a variable $\avar_j$ taking a value in the interval
$\texttt{I}_j^i$. Identifying a satisfying interpretation of the complete 3-CNF is related to the study of random interval
graphs~\cite{Scheinerman88,JusticzScheinermanWinkler90}. Our notation and terminology on signed SAT follows~\cite{ChepoiCreignouHermannSalzer10}.

Signed SAT problems originated in the area of so-called multi-valued logic~\cite{lukasiewicz:20}, where variables can take a (usually finite) number of so-called
\textit{truth values}, not just $\TRUE$ or $\FALSE$.  Work on signed CNF formulas started in earnest with the work of H\"ahnle and Many\`a and their
coauthors.  We refer the reader to the survey paper~\cite{BeckertHaehnleManya99}, and the references therein.

The motivation for studying signed formulas was to extend algorithmic techniques developed for deductive systems in multi-valued logic to better cover practical
applications~\cite{haehnle:91}.  Indeed, on the one hand, a number of papers show how combinatorial problems can be solved using signed SAT
algorithms~\cite{BejarManya99,BejafCabiscolFernandezManyaGomes,FrischPeugniez01,BejarManyaCabiscolFernandezGomes07}; on the other hand, a large number
of heuristic and exact algorithms have been studied (see~\cite{AnsoteguiManya03, britz04} and the references therein), and a number of polynomially solvable
subclasses of signed SAT have been
identified~\cite{EscaladaimasManya94,BeckertHaehnleManya99,Manya00,BeckertHahnleManya00,AnsoteguiBejarCabiscolManya04,AnsoteguiManya03,ChepoiCreignouHermannSalzer10}.
While in the works of Many{\`a} and his collaborators, order-theoretic properties of the ground set $M$ are exploited to make conclusions on the complexity of signed
SAT, Chepoi et al.~\cite{ChepoiCreignouHermannSalzer10} completely settle the complexity question in the general case by reverting to combinatorial properties of the
set system $\mathcal S$.  In particular, they prove that:
signed $k$-SAT, $k\ge 3$, is polynomial, if $\bigcap_{S\in\mathcal S} S \ne \emptyset$ and NP-complete otherwise;
signed 2-SAT is polynomial if, and only if, $\mathcal S$ has the Helly property (if no two sets in a subfamily are disjoint, then the subfamily has non-empty
intersection), and NP-complete otherwise.

For the case when $\mathcal S$ has the Helly property, Chepoi et al.~give a non-satisfiability certificate for signed 2-SAT in the spirit of Aspvall, Plass, and
Tarjan's famous result for classical 2-SAT~\cite{AspvallPlassTarjan79}.

Most applications and a great deal of the earlier complexity results~\cite{BeckertHaehnleManya99} focus on \textit{regular} signed SAT, where $M$ is a poset, and the
formulas may only involve sets of the form $S = \{j\mid j\ge i\}$ or $S = \{j\mid j\le i\}$.  Regular iSAT (or just regular SAT) is regular signed SAT for posets $M$
which are chains.

In particular, for regular iSAT, random formulas have been investigated from a heuristic point of view.
Many\`a et al.~\cite{ManyaBejarEscaladaimaz98} study uniformly generated random regular 3-iSAT instances,
and observe a phase transition similar to that observed in classical SAT (see~\cite{AchlioptasPeres2004}
and the references therein): (i) the most computationally difficult instances tend to be found near the
threshold, (ii) there is a sharp transition from satisfiable to unsatisfiable instances at the threshold
and (iii) the value of the threshold increases as the number of truth values considered increases.  Their
results are confirmed and extended by further papers exploring uniformly random regular 3-iSAT
instances~\cite{BejarManya99phase,BeckertHaehnleManya99,BejarManyaCabiscolFernandezGomes07}.

Further, in~\cite{BejarManya99phase,BejarManyaCabiscolFernandezGomes07} a bound on the ratio $\nfrac{m}{n}$ is given, beyond which a random formula is with high
probability (whp) unsatisfiable.  To our knowledge, however, ours is the first rigorous analysis of an algorithm for random signed SAT.

\paragraph{Our interest} in the particular version of signed SAT arises from applications in computational systems biology, where iSAT yields a generalization of
modeling with Boolean networks~\cite{kauffman:69},
where biological systems are represented by logical formulas with variables corresponding to biological components like proteins.  Reactions are modeled as logical
conditions which have to hold simultaneously, and then transferred into CNF.  The model is widely used by practitioners (see e.g.~\cite{downward:2001,
  klamt-saezr-lindquist-simeoni-gilles:2006, haus-niermann-truemper-weismantel:2009} and the references therein). Often, though, this binary approach is not
sufficient to model real life behavior or even accommodate all known data.  Due to new measurement techniques, a typical situation is that an experiment yields
several ``activation levels'' of a component.  Thus, one wants to make statements of the form: If the quantity of component $A$ reaches a certain threshold but does
not exceed another, and component $B$ occurs in sufficient quantity, then another component $C$ is in a certain frame of activation levels.  The collection of such
rules accurately models the global behavior of the system.  We refer to~\cite{KBallersteinPhD} for details of models and applications.


\paragraph{In this paper}
we present and analyze an algorithm which solves uniformly random 3-iSAT instances with high probability,
provided that the ratio between the number $m$ of clauses and the number $n$ of variables is at most
\mbox{2.3}.  Our algorithm is an adaption of the well-known Unit Clause algorithm from classical
SAT~\cite{ChaoFranco86,Achlioptas01}, where, in an inner loop, 1-clauses are treated if any exist, and in
an outer loop, a variable is chosen freely and assigned some value.
This Unit Clause approach is enhanced with a ``repair'' subroutine (a very simple backtracking mechanism).

The algorithm in~\cite{FriezeSuen96} is currently the best known algorithm that succeeds with high
probability, although other algorithms (e.g., \cite{KaporisKirousisLalas06,HajiaghayiSorkin03}) can be
outfitted with a backtracking routine to provide better results.  See also~\cite{CojaOghlan10} for general
$k\to\infty$.

Unlike the algorithms in~\cite{Achlioptas00,AchlioptasSorkin00,KaporisKirousisLalas06}, we prove that our
algorithm succeeds with high probability.  To obtain a whp result, the ``repair'' subroutine is essential,
cf.~e.g., \cite{FriezeSuen96}, where the range in which their algorithms succeed increase dramatically,
once such a routine part is added.  As for our algorithm, without such a repair function, it would not
succeed whp, if the ratio $m/n$ is larger than the point where 2-iSAT formulas become satisfiable almost
surely.  This mirrors the situation in classical $3$-SAT~\cite{Achlioptas01} (and can also be derived from
our analysis).

In the case of iSAT, the repair mechanism needs to be considerably more subtle than the one
in~\cite{FriezeSuen96} for classical 3-SAT.

In the analysis of the algorithm, we use Wormald's differential equations method~\cite{Wormald95}.  ODE
methods have been used for the analysis of algorithms for classical SAT with great
success~\cite{ChaoFranco86,ChaoFranco90,FriezeSuen96,Achlioptas00,AchlioptasSorkin00}.
In our analysis, we combine the idea of Achlioptas and Sorkin~\cite{AchlioptasSorkin00} to consider as a
time step an iteration of the outer loop, but we use Wormald's theorem~\cite{Wormald-LecNot-99} where they
use a Markov-chain based approach.  The analysis of the inner loop requires to study the first busy period
of a certain stable server system~\cite{Achlioptas00,Achlioptas01}, or, in our case, more accurately, the
total population size in a type of branching process.  The value \mbox{2.3} arises from the numerical
solution to an initial value problem (IVP).

At this point, it makes sense to point to the fact that while, in general, backtracking destroys uniform
randomness of the formula, which is problematic for the analysis.  In our analysis, (1) the repair
involves only a very small part of the formula---what remains of the formula is still uniformly
distributed---so that (2) a more careful analysis is only needed for what happens in the repaired part of
the formula itself.

Extending the results for $k$-iSAT for $k\ge 4$ is conceptually easy; we briefly discuss it in the
conclusions.

\paragraph{The outline of the paper is as follows:}
In the next section, we present our algorithm for random 3-iSAT in detail.  In
Section~\ref{sec:cmprndint}, we prove some facts about uniformly at random chosen sub-intervals of
$[0,1]$.  In Section~\ref{sec:2isat} we take a brief excursion to random 2-iSAT as our algorithm for
3-iSAT ultimately relies on solving a 2-iSAT instance.  In Section~\ref{sec:Q}, we compile the required
facts about total population sizes of a kind of branching system, which are then applied in
Section~\ref{sec:inner} to the study of the inner loop of our algorithm.  Finally, in
Section~\ref{sec:outer}, we prove the whp result for our algorithm.  We raise some issues for future
research in the final section.  Several technical arguments have been moved into the appendix.

\paragraph{Throughout the paper,} 
we hide absolute constants in the big-$O$-notation.  If the constant depends on other parameters, we make this clear by adding an index, e.g.,~$O_\eps(\cdot)$.  As
customary, we use the abbreviation iid for ``independent and identically distributed'' and uar for ``uniformly at random''.  Whp and wpp are to be understood for
$n\to\infty$, with $m=m(n)$ depending on $n$.


\section{An algorithm for random 3-iSAT}\label{sec:3-isat}

In this section, we describe an algorithm which finds a satisfying interpretation if the number of clauses is $m = cn$ with $c \le 2.3$.

\subsection{The random model; exposure}

For our random model, we assume that each 3-clause consists of three distinct variables. We choose a formula uar from the set of all possible classical 3-CNF
formulas on $n$ variables with $m$ 3-clauses, each containing three distinct variables. Then, we choose an interval for each literal uar from the subintervals of
$[0,1]$: We select uar two points $x$ and $y$ from $[0,1]$ and determine the interval as $[a,b]$ with $a=\min\{x,y\}$ and $b=\max\{x,y\}$. In this context, note
that due to Scheinerman~\cite{Scheinerman88} the endpoints $x$ and $y$ can be arbitrary reals. In fact, he proves that this strategy is equivalent to choosing $2l$
endpoints for $l$ intervals uar from the finite set $\{1,\ldots,2l\}$ without repetition as the probability that all chosen endpoints from $[0,1]$ are distinct is
$1$. For the distribution of a random interval $[a,b]$ chosen as $a=\min\{x,y\}$ and $b=\max\{x,y\}$ for $x,y\in[0,1]$ uar, we find with $u,v\in [0,1]$
\begin{align*}
  \Prb ([a,b]\subseteq [u,v])= 2\cdot \Prb (a\geq u,\ b\leq v) = 2\cdot (1-u)\cdot v.
\end{align*}

 As is customary in the context of random SAT, we use the language of ``exposing'' literals.  Intuitively, the idea is that the information
about each literal is written on a card which lies face down, until the information is exposed.  Clearly, the unexposed part of the formula is uar conditioned on
which literals have been exposed and which have not. We refer to the elegant description in Achlioptas' paper~\cite{Achlioptas01}.

\subsection{Brief description of the algorithm}

The basic framework of our algorithm is the same as for most algorithms for classical $k$-SAT.  A formerly unused variable is selected, and a value is assigned to
it.  Then, clauses containing the variable are updated: if the literal of the clause involving the variable is satisfied, the clause is deleted; otherwise the
literal is deleted from the clause, leaving a shorter clause.  The variable is removed from the set of \textit{unused variables,} and declared a \textit{used
  variable.}  The algorithm fails if, and only if, it creates an empty clause.

However, to a certain extent, our algorithm is able to repair bad choices it has made.  Thus, it occasionally only assigns \textit{tentative} values to variables.
As long as it is not certain that a variable keeps its tentative value, no deletions of clauses or literals from clauses are performed.  Instead, we assign colors
to the clauses, which code the number of satisfied, unsatisfied, and unexposed literals they contain.  The meaning of the colors will be explained in
Table~\ref{tab:colors} but at this point it suffices to know that red clauses correspond to unexposed 1-clauses, i.e., clauses with one unexposed literal and the
variables in any other literal of the clause have tentative values which render the literals false.

As said before, the basic approach is that of the Unit-Clause algorithm.  The \textit{outer loop} of the algorithm will maintain the property that there is no
1-clause.  In each iteration of the outer loop, a variable is selected uar from the set of unused variables.  Such a variable selected in the outer loop is
referred to as a \textit{free variable.}  The \textit{inner loop} is initialized by assigning a tentative value to this free variable, and then repeats as long as
there are red clauses.  In each iteration of the inner loop, a red clause is selected and \textit{serviced:} the variable contained in the clause (the
\textit{current variable} of the iteration) is tentatively set to some value in such a manner that the serviced red clause becomes true.  We refer to the variables
selected in the inner loop as \textit{constrained variables.}

If, during a run of the inner loop, a situation is reached in which it is probable that an empty clause will be created, it backtracks.  This happens when the
following \textit{fatality} is suffered: The current variable occurs in another red clause, other than the one serviced.  If that happens, there is a $\nfrac13$
probability that the two intervals occurring in the two red clauses are disjoint~\cite{Scheinerman88}, so that creating an empty clause is inevitable.

For this situation, the inner loop maintains a rooted tree $G$ of decisions it has taken so far. The nodes of the tree correspond to variables to which tentative
values have been assigned and those which occur in the unexposed part of red or blue $2$-clauses. The root of the tree is the free variable with which the run of
the inner loop was initialized. The edges correspond to $2$-clauses. For every $2$-clause in which the current variable of an iteration occurs, the unexposed
variable is added as a node and an edge is added connecting the current variable with this new node. Doing so in every iteration constructs a tree. If the current
variable of an iteration occurs in two red clauses, then this implies that a cycle is closed in $G$, because there must exist two paths from $\avar_0$ to the
current variable. The tree $G$ is in detail defined in the algorithm. If a fatality occurs, the values of the variables along the paths from the root to the
serviced literal are changed so that all $2$-clauses along the path are fulfilled and only one red clause remains which is satisfied.  Then, all other tentative
values are made permanent, and the inner loop is restarted with the new formula, but this time without a free variable in the initialization.  We call
\textit{Phase~I} the run of the inner loop before a repair occurs (or if no repair occurs), and as \textit{Phase~II} to the run of the inner loop after a repair
has been performed.  In Phase~II, no further repair is attempted.  Instead, if fatalities occur, the inner loop just moves on (without repair).  In Phase~I, if a
fatality occurs, there's the possibility that a repair is not possible.  In this case, too, the inner loop just moves on without repair.  In order to be able to
\newcommand{\hurz}{``raise a flag''}%
\newcommand{\hurzes}{``raises a flag''}%
\newcommand{\hurzed}{``raised a flag''}%
refer to these situations in the proofs, we indicate these positions in the code by the pseudo-command \hurz.

After all red clauses have been dealt with in either Phase~I or Phase~II, the tentative values are made permanent, and control is returned to the outer loop, which
selects another free variable, and so on.

The outer loop terminates, if the number of 2-clauses plus the number of 3-clauses drops below a certain factor $c'$ of the number of unused variables.  Then, it
deletes an arbitrary literal from every 3-clause and invokes the exact polynomial algorithm by Chepoi et al.~\cite{ChepoiCreignouHermannSalzer10} to decide whether
the resulting 2-iSAT formula has a satisfying interpretation.  We will prove in Section~\ref{sec:2isat} that this is always the case if the ratio of the number of
resulting 2-clauses over the number of unused variables is below $\frac32$.

The complete algorithm is shown below as Algorithm~\ref{algo:3isat:outer} (the outer loop), Algorithm~\ref{algo:3isat:inner} (the inner loop), and
Algorithm~\ref{algo:3isat:repair} (the repair procedure).
Throughout the course of the algorithm, for $i=0,1,2,3$, we denote by $Y_i(t)$ the number of $i$-clauses, and by $X(t)$ the number of unused variables,
respectively, at the beginning of iteration $t$ of the outer loop.  Moreover, for an interval $I$, we denote by
\begin{equation}\label{eq:3isat:def-bar-x}
  \bar x(I) := \mbox{argmin}_{x\in I} \,\sabs{x-\nfrac12}
\end{equation}
the point in $I$ which is closest to $\nfrac12$.
We refer to the variable $\avar_j$ which is selected in iteration~$j$ of the inner loop as the \textit{current variable} of that iteration.

Below, we will prove the following fact.

\begin{lemma}\label{lem:3isat:failure}
  A single run of Algorithm~\ref{algo:3isat:inner} (including a possible repair and consequent Phase~II) produces an empty clause, only if it \hurzes.
\end{lemma}

The performance of the algorithm on random 3-iSAT instances is analyzed in Sections~\ref{sec:inner} and~\ref{sec:outer}.  There, we will prove the following theorem.

\begin{theorem}\label{thm:algo-works-whp}
  Let $c := 2.3$, and suppose Algorithm~\ref{algo:3isat:outer} is applied to a uniformly random iSAT formula on $n$ variables with $m$ 3-clauses. If $m \le cn$,
  then, whp, Algorithm~\ref{algo:3isat:outer} creates no empty clause, i.e., it finds a satisfying interpretation.
\end{theorem}

The value $2.3$ is determined through the numerical solution of an initial value problem.  It corresponds to the point in which the increase in red clauses in each
iteration of the inner loop would become so large that the inner loop will not terminate.


\newcommand{\refouter}[1]{(o-\ref{#1})}
\newcommand{\refouternoparen}[1]{o-\ref{#1}}
\newsavebox{\looop}
\sbox{\looop}{%
  \begin{minipage}{1.0\linewidth}\flushleft%
    \begin{enumerate}[({o-{\ref{algstep:nb9ds8bf98estr}}}.1)]
    \item Choose a variable $\avar$ uar.
    \item Invoke {\it Inner loop} (Phase I).
    \item $t := t+1$
    \end{enumerate}%
  \end{minipage}%
}

\begin{algorithm}[htp]\flushleft%
  \begin{enumerate}[({o-}1)]
  \item\label{algstep:inner:start} Given: 3-CNF-formula; positive constant $c'$.
  \item $t := 0$
  \item\label{algstep:nb9ds8bf98estr} While $Y_2(t) + Y_3(t) > c' X(t)$:\\[.2\baselineskip]
    \usebox{\looop}
  \item In every $3$-clause, remove one literal at random.
  \item Invoke Chepoi et al.'s algorithm (cf.~Section~\ref{sec:2isat}) for the remaining 2-iSAT formula.
  \end{enumerate}
  \caption{\quad\it UC w/ backtracking (outer loop)}\label{algo:3isat:outer}
\end{algorithm}%


\begin{table}[htb]
  \centering
  \begin{tabular}[t]{lp{.85\textwidth}}
    \toprule
    Color     & Meaning \\
    \midrule
    Uncolored & All literals in the clause are unexposed.\\
    Black     & All literals are exposed. \\
    Red       & The clause has precisely one unexposed literal.   The tentative values of any other variables in the clause make the corresponding literals false.
    In particular, unexposed 1-clauses are red. \\ 
    Blue      & The clause contains precisely one unexposed literal and at least one exposed literal which evaluates to true for the tentative value of its variable.\\
    Pink      & The clause is a 3-clause, precisely one of its literals is exposed, and this literal evaluates to false for the tentative value of its variable.\\
    Turquoise & The clause is a 3-clause, precisely one of its literals is exposed, and this literal evaluates to true for the tentative value of its variable.\\
    \bottomrule
  \end{tabular}
  \caption{Semantics of the colors of the clauses.}\label{tab:colors}
\end{table}

\newcommand{\refinner}[1]{(i-\ref{#1})}
\newcommand{\refinnernoparen}[1]{i-\ref{#1}}
\newsavebox{\ilooop}
\savebox{\ilooop}{%
    \begin{minipage}{.89\linewidth}\flushleft%
      \begin{enumerate}[{(\refinnernoparen{algstep:inner:red1cl}.}1)]
      \item\label{algstep:inner:hyperedge:obacht} If there is a red, blue, or black 3-clause: 
        \hurz!
      \item\label{algstep:inner:cycle:obacht} If the graph $G$ contains a cycle, or $\avar_j$ is in a blue clause: \hurz!
      \item\label{algstep:inner:double-hit:obacht} If $\avar_j$ occurs in three or more red clauses (including $\aclause_j$): \hurz!
      \item\label{algstep:inner:repair-situ:obacht} Otherwise: \textbf{Phase~I} is completed.
        Let $\aclause'$ be the unique red clause different from $\aclause_j$ containing $\avar_j$ in a literal $\aliteral' = \avar_j\in \aJ'$.  Repair the unique
        path between~$\avar_0$ and~$\aclause_j$; then initiate \textbf{Phase II}.
      \end{enumerate}%
    \end{minipage}%
}

\begin{algorithm}[htp]\flushleft%
  \begin{enumerate}[({i}-1)]
  \item Given:
    \begin{itemize}
    \item In \textbf{Phase I}: formula consisting of 2- and 3-clauses only; a (free) variable $\avar_0$.
    \item In \textbf{Phase II}: formula consisting of 1-, 2- and 3-clauses.
    \end{itemize}
  \item $j:=0$
  \item Initialize:\label{algstep:inner:AI}
   Expose the occurrences of $\avar_0$ in all clauses.
    \begin{itemize}
    \item In \textbf{Phase I} only:
      \begin{enumerate}[({\refinnernoparen{algstep:inner:AI}}.1)]
      \item Tentatively set $\avar_0$ to $\nfrac12$.
      \item Initialize the graph $G := (\{\avar_0\},\emptyset)$.
      \end{enumerate}
    \item In \textbf{Phase II} only:
      \begin{enumerate}[({\refinnernoparen{algstep:inner:AI}}.1)]
      \item Color all 1-clauses red.
      \end{enumerate}
    \end{itemize}
  \item Expose the intervals associated with $\avar_0$. Color clauses containing $\avar_0$ according to Tab.~\ref{tab:colors}.
  \item $j := j+1$ \label{algstep:inner:goto-anchor}
  \item If there is no red clause,\label{algstep:inner:while} exit inner loop: Set all variables to their tentative values; remove satisfied clauses and remove
    violated literals from their clauses; return to outer loop.
  \item\label{algstep:inner:current-var} Select a red clause $\aclause_j$ at random; let $\anatom_j$ be the unexposed literal in $\aclause_j$; expose current
    variable $\avar_j$ of $\anatom_j$
  \item\label{algstep:inner:expose-occ-colored} Expose all occurrences of $\avar_j$ in colored clauses.
  \item If $\avar_j$ is contained in a red clause other than $\aclause_j$:\label{algstep:inner:red1cl}
    \begin{itemize}
    \item In \textbf{Phase~I} only:\\[.2\baselineskip]%
      \usebox{\ilooop}\\[.3\baselineskip]%
    \item In \textbf{Phase II} only: \hurz!
    \end{itemize}%
  \item\label{algstep:inner:hit-remaining-occ} Expose all occurrences of $\avar_j$ in all uncolored clauses.
  \item For every uncolored 2-clause $\avar_j \in I \lor \anothervar \in J$ containing $\avar_j$, add to $G$ the vertex $\anothervar$ and the edge \mbox{$\avar_j
      \in I \lor \anothervar \in J$} \quad between $\avar_j$ and $\anothervar$.
  \item Tentatively set $\avar_j$ to $\bar x(\anI_j)$.
  \item\label{algstep:inner:colorize} Update the colors of all clauses containing $\avar_j$.
  \item Goto step~\refinner{algstep:inner:goto-anchor}.
  \end{enumerate}
  \caption{\quad\it Inner loop}\label{algo:3isat:inner}
\end{algorithm}


\newcommand{\refrepair}[1]{(r-\ref{#1})}
\newcommand{\refrepairnoparen}[1]{r-\ref{#1}}
\begin{algorithm}[htp]\flushleft%
  \begin{enumerate}[({r}-1)]
  \item Given: Set of colored 1-, 2- and 3-clauses; a literal $\aliteral' = \avar_k\in \aJ'$; a path of the form\\
    $\avar_0$, \quad 
    $\avar_0\in \aJ_0 \lor \avar_1\in \anI_1$, \quad
    $\avar_1\in \aJ_1 \lor \avar_2\in \anI_2$, \quad
    $\dots$, \quad
    $\avar_{k-1}\in \aJ_{k-1} \lor \avar_k\in \anI_k$;
  \item For $j=0,\dots,k-1$:\label{algstep:set-loop}
    \begin{enumerate}[({\refrepairnoparen{algstep:set-loop}}.1)]
    \item Set $\avar_j$ (permanently) to $\bar x(\aJ_j)$
    \end{enumerate}
  \item\label{algstep:repair:x-k} Set $\avar_k$ (permanently) to $\bar x(\aJ')$
  \item\label{algstep:repair:finalize} Set all variables from Phase~I, except those which have just been set in~\refrepair{algstep:set-loop} and
    \refrepair{algstep:repair:x-k}, to their tentative values; remove satisfied clauses and remove violated literals from their clauses.
  \end{enumerate}
  \caption{\quad\it Repair path}\label{algo:3isat:repair}
\end{algorithm}

\subsection{Comparison to algorithms for classical SAT}

For classical SAT, if a variable $\avar$ is set to a value, the probability that a random literal containing $\avar$ evaluates to true is $\nfrac 12$ ---
independent of the value.  As will become apparent in the next section, this is far from true for random interval literals.  There, the value $\nfrac 12$ is the
single, most likely value to be contained in a random interval (the probability is $\nfrac12$) and all other values are less likely. Hence, we will assign
$\nfrac12$ to the variables as long as possible which is for all free variables.

The rationale behind assigning the value $\nfrac12$ to free variables is two-fold.
Firstly, it makes the analysis a lot more easy than if one tries to find a maximum cardinality subset of literals containing $\avar$ all of whose intervals have
pairwise non-empty intersection.
Secondly, for large numbers of literals containing $\avar$, the maximum cardinality of a subset with pairwise intersecting intervals is asymptotically attained by
taking all literals with intervals containing $\nfrac12$ (this is Theorem~4.7 of Scheinerman's paper~\cite{Scheinerman88}). This, in particular, implies that
assigning an interval of values to a variable does asymptotically not lead to a satisfying interpretation of the formula which is not satisfying if assigning the
single value $\nfrac12$.

The situation for constrained variables is similar, but a bit more complicated.  For constrained variables, we are free only to choose the value for the variable
within the interval $\anI$ for the literal $\aliteral = \avar \in \anI$ which we wish to satisfy.  Unlike to classical SAT, where this does not change the
probability that other random literals containing $\avar$ are satisfied, depending on $\anI$, this probability may change considerably.  Moreover, for two literals
containing $\avar$, the two events of both being satisfied simultaneously with $\aliteral$ are not independent.

However, an adaption of Scheinerman's argument mentioned above shows that, asymptotically, the best choice is to take the point $\anI$ which is closest to
$\nfrac12$ as we do in our algorithm.

Concerning the backtracking part of the algorithm, we would like to point out the difference to the approach in~\cite{FriezeSuen96}.  If the (essentially
identical) fatality is suffered, a very elegant remedy is to simply flip the values of all variables with tentative values: if the tentative value of a variable is
\TRUE, make it \FALSE, and vice versa.  Needless to say, for variable values in a larger set, there is no obvious choice for the new value of a variable.  Thus, in
our approach, we have to choose the variable values in a smart manner, with the single aim to undo the fatality. Namely, those variables that led to the fatality
are assigned $\bar x(I)$ as described in \emph{Repair Path} (Algorithm~\ref{algo:3isat:repair}).

\subsection{Proof of the \hurz-lemma}

\begin{proof}[Proof of Lemma~\ref{lem:3isat:failure}]
  Assume that Algorithm~\ref{algo:3isat:inner} does not \hurz.

  The only place where a 0-clause can be generated without having \hurzed\ is in the final step~\ref{algstep:repair:finalize} of the repair,
  Algorithm~\ref{algo:3isat:repair}.  Clearly, none of the clauses on the path will become empty.
  
  Moreover, setting the final variable, $\avar_k$, cannot create an empty clause, because of the conditions in
  steps~(\refinnernoparen{algstep:inner:red1cl}.\ref{algstep:inner:hyperedge:obacht})
  and~(\refinnernoparen{algstep:inner:red1cl}.\ref{algstep:inner:cycle:obacht}).

  For a 3-clause to become empty, it is necessary that when the repair is invoked in Algorithm~\ref{algo:3isat:inner}, all three of its literals have been exposed
  (possibly in the same iteration).  In other words, it must have been red, blue, or black in step
  (\refinnernoparen{algstep:inner:red1cl}.\ref{algstep:inner:hyperedge:obacht}), a contradiction.
  
  For a 2-clause to become empty, both literals must have been exposed, one of them possibly in the iteration where the repair occurs. Moreover, if it was blue,
  the value of the variable satisfying one of its literals must change during the repair.
  In other words, the following three scenarios are possible:
  \begin{enumerate}[\it(i)]
  \item it was black before the repair was invoked
  \item it was red before the repair was invoked, but it contains $\avar_j$
  \item it was blue before the repair was invoked, it is of the form $\avar_i \in \anI_i \lor \avar_j\in \anI_j$ for some $i<j$, and $\avar_i$ is one of the variables
    set in step~\refrepair{algstep:set-loop} of Algorithm~\ref{algo:3isat:repair}.
  \end{enumerate}

  In case~(i), if the black 2-clause becomes an empty clause, either it was red when its final literal was exposed, a contradiction, or it was blue, which means
  that at least one of its variables lies on the path which is repaired. If the whole clause lies on the path, we have already noted that it cannot become empty.
  If only one of its variables is on the path, then it must be an edge in the tree having one end vertex on the path and the other lying further away from the root
  than the path. The fact that it is black means that the variable which is not on the path was the current variable of some earlier iteration $i<j$.  But then the
  corresponding literal was either the selected literal $\aliteral_i$, in which case it was satisfied by the tentative value of $\avar_i$, or the if-condition in
  step~\refinner{algstep:inner:red1cl} for iteration $i$ held, which is a contradiction (either a repair occurred, or the algorithm has \hurzed).
  
  In case~(ii), if the 2-clause is on the path, it does not become empty.  If it is the unique other red clause $\aclause'$, then it will be satisfied in the
  initialization of Phase~II.

  Case~(iii), is not possible because of the condition in step~(\refinnernoparen{algstep:inner:red1cl}.\ref{algstep:inner:cycle:obacht})
\end{proof}

\subsection{Random formulas}

The following easy facts (see the discussion at the beginning of this section) underly the analysis of the
algorithm on random formulas.

\begin{lemma}\label{lem:3isat:uar}
  If Algorithm~\ref{algo:3isat:outer} is invoked with a uar random 3-iSAT formula, then
  \begin{enumerate}[(a)]
  \item at the beginning of each iteration of the outer loop, the current formula is distributed uar conditioned on the number of unused variables, 2-clauses,
    and 3-clauses;
  \item at the beginning of each iteration of the inner loop, the current formula is distributed uar conditioned on the number of unused variables, 1-clauses,
    2-clauses, 3-clauses, and the colors of the clauses.
  \item at the beginning of Phase~II in the inner loop, the current formula is distributed uar not only conditioned on the number of unused variables, 1-clauses,
    2-clauses, 3-clauses, the colors of the clauses, and the list $L$ of clauses which are known not to contain $\avar_0$ and the list of clauses in which an
    occurrence of $\avar_0$ has been exposed.
  \end{enumerate}
\end{lemma}

By Lemma~\ref{lem:3isat:uar}, the history of the random process defined by the outer loop, that is, for each $t$, the state of the formula and all other
information relevant to how the algorithm will proceed, available at the beginning of iteration $t$, is completely determined by
\begin{equation}
  \Hist(t) := (X(t),Y_2(t),Y_3(t));
\end{equation}
in particular it is Markov.


\section{Computations for random intervals}\label{sec:cmprndint}
In this section, we make some computations regarding intervals chosen uar from the subintervals of $[0,1]$ as described before. We refer to
\cite{Scheinerman88,JusticzScheinermanWinkler90} for further background.

We aim to study the event $\bar x(I) \in J$, with two random intervals $I$ and $J$ ($\bar x$ is defined in~\eqref{eq:3isat:def-bar-x}).  We start with the
following observation.

\begin{lemma}[\cite{Scheinerman88}]\label{lem:rint:x-in-I}
  For $x \in [0,1]$ and for a random interval $I$, we have
  \begin{equation*}
    \Prb[ x\in I ] = 2 x(1-x).
  \end{equation*}
  In particular, the probability that a random interval contains the point $\nfrac12$ is $\nfrac12$.
\end{lemma}

The cumulative distribution function of $\bar x(I)$ can be written down.

\begin{lemma}\label{lem:rint:close-to-half}
  For a random interval $I$, the random variable $\bar x(I)$ has cumulative distribution function
  \begin{equation}\label{eq:rint:distrib-cth}
    F(t) := \begin{cases}
      0           & \text{ if } t \le 0\\
      t^2,        & \text{ if } t < \nfrac12 \\
      1-(1-t)^2,  & \text{ if } t \ge \nfrac12\\
      1           & \text{ if } t \ge 1.
    \end{cases}
  \end{equation}
\end{lemma}
\begin{proof}
  Direct computation.
\end{proof}

Let $X$ be a random variable with cumulative distribution function $F$ as in~\eqref{eq:rint:distrib-cth}, and define
\begin{equation}\label{eq:Q:def-P}
  P := 1-2X(1-X).
\end{equation}
Thus, by the previous two lemmas, for the probability that, for two random intervals $I$ and~$J$ we have
$\bar x(I) \in J$, we have
\begin{equation*}
  \Prb[ \bar x(I) \in J ]
  = \Exp( \Prb[ X \in J \mid P ] )
  = \Exp( 1-P )
  = 1 - \Exp P.
\end{equation*}

The following computations are straightforward, see~\ref{apx:rint:moments-of-P}.
\begin{lemma}\label{lem:rint:moments-of-P}\mbox{}%
  \begin{enumerate}[(a)]
  \item\label{lem:rint:moments-of-P:mean}      $\displaystyle \Exp P   = \nfrac{13}{24}$
  \item\label{lem:rint:moments-of-P:2ndmoment} $\displaystyle \Exp P^2 = \nfrac{3}{10}$
  \qed
  \end{enumerate}
\end{lemma}

\begin{lemma}\label{lem:rndint:middlemost-point-in-other}
  For two random intervals $I,J$, the following is true.
  \begin{equation*}
    \Prb[ \bar x(I) \in J ] = \frac{11}{24}.
  \end{equation*}
\end{lemma}
\begin{proof}
  Immediate from Lemmas \ref{lem:rint:x-in-I}, \ref{lem:rint:close-to-half}, and~\ref{lem:rint:moments-of-P}(a).
\end{proof}

\begin{remark}
  It could be interesting to choose the intervals in a different way rather than uniformly at random, for instance, to reflect certain realistic
  structures.  However, the strategy of choosing intervals does not change the main analysis of the algorithm.  The only adaptions to be made are the previous
  computations of the probabilities, and thus the new constants need to be used in the analysis, which can lead to different results.
\end{remark}

\newcommand{\x}{\texttt{x}}
\newcommand{\y}{\texttt{y}}
\newcommand{\I}{\texttt{I}}
\newcommand{\J}{\texttt{J}}
\newcommand{\K}{\texttt{K}}

\section{2-iSAT}\label{sec:2isat}\label{sec:tarjan}
In this section, we take a brief glance at the situation for random 2-iSAT.  The reason is that, ultimately, our 3-iSAT algorithm reduces the 3-iSAT formula to one
with exactly two literals per clause, and then invokes the polynomial time algorithm by Chepoi et al.~\cite{ChepoiCreignouHermannSalzer10} to find a solution.  We
need to make sure that the resulting random 2-iSAT instance is satisfiable.

For this, we proceed along the same lines as~\cite{ChvatalReed92}, using Chepoi et al.'s Aspvall-Plass-Tarjan-type~\cite{AspvallPlassTarjan79} certificate for the
non-satisfiability of signed 2-SAT formulas for set systems satisfying the Helly-property.  We describe the certificate now.

For a 2-iSAT formula $F$, define a digraph $G_F$ which contains two vertices labeled $\avar\anI t$ and $\avar\anI f$, respectively, for every literal $\avar\in I$
occurring in $F$.  For every clause $\avar \in \anI \lor \avar' \in I'$ of $F$, the digraph $G_F$ contains two arcs $\avar\anI f \to \avar'\anI' t$ and $\avar'\anI'
f \to \avar\anI t$.  We refer to these arcs as \textit{clause arcs}.  Moreover, for every two literals $\avar\in\anI$ and $\avar\in\aJ$ occurring in $F$, if $\anI \cap
\aJ = \emptyset$, the digraph $G_F$ contains the two arcs $\avar\anI t \to \avar\aJ f$ and $\avar\aJ t \to \avar\anI f$.  These arcs we call \textit{disjointness
  arcs.}

For a literal $\avar\in I$ occurring in $F$, we refer to the vertex $\avar\anI t$ as a \textit{positive} vertex, and to $\avar\anI f$ as a \textit{negative} vertex.
Moreover, we say that these two vertices are \emph{complements} of each other; in other words, the complement of the (positive) vertex $\avar\anI t$ is the (negative)
vertex $\avar\anI f$ and vice versa.  Note that arcs originating from negative vertices are clause arcs, while arcs originating from positive vertices are
disjointness arcs.

Chepoi et al.\ relate the satisfiability of $F$ to the strongly connected components (\textit{SCC}s) of $G_F$.

\begin{proposition}[Aspvall-Plass-Tarjan-type certificate,~\cite{ChepoiCreignouHermannSalzer10}]\label{prop:2isat:Chepoi}
  The formula $F$ is satisfiable if, and only if, no SCC of $G_F$ contains a pair of vertices which are complements of each other.
\end{proposition}

\begin{remark}
  A path in $G_F$ of length $\ell$ contains $\lfloor l/2 \rfloor$ or $\lceil l/2 \rceil$ disjointness arcs, and no two of them are incident.
\end{remark}

Chepoi et al.\ also give an algorithm which determines, in polynomial time, whether a formula $F$ is satisfiable, and if it is, produces a satisfying
interpretation.  We refer to their paper for details.

From Proposition~\ref{prop:2isat:Chepoi}, we obtain the following corollary.

\begin{corollary}\label{cor:pretzel}
  If $F$ is not satisfiable, then $G_F$ contains a \textit{bicycle,} i.e., a directed walk
  \begin{equation*}
    u_0\rightarrow\dots\rightarrow u_{\ell+1},
  \end{equation*}
  with at least one clause arc, and the following properties:
  \begin{enumerate}[(a)]
  \item the literals in the vertices $u_1,\dots,u_\ell$ are all distinct;
  \item the literals in the vertices $u_0$ and $u_{\ell+1}$ occur among the literals in the other vertices;
  \item the clauses in the arcs are all distinct.
  \end{enumerate}
\end{corollary}
\begin{proof}
  For a vertex $v$, we denote its complement by $\bar v$.  By what we said about the different types of arcs, on every path from $v$ to $\bar v$, there is at least
  one clause arc.

  Choose an SCC and take a pair of complementing vertices $v$ and $\bar v$ in the SCC such that the distance from $v$ to $\bar v$ in $G_F$ is minimal.  Then, on the
  shortest path $P$ from $v$ to $\bar v$, no literal appears twice.  Denote by $\aliteral$ the literal defining $v$ and $\bar v$.
  
  Now take a shortest path $Q$ in $G_F$ form $\bar v$ to $v$.  If there is no literal other than $\aliteral$ which appears twice on $P\cup Q$, then $P\cup Q$ is a
  bicycle starting and ending in $v$.  On the other hand, if there is a literal $\aliteral'$ other than $\aliteral$ which appears twice on $P\cup Q$, then the
  desired bicycle is constructed by taking the path $P$ from $v$ to $\bar v$, and then the path $Q$ until the first vertex whose literal already occurred earlier.
\end{proof}

Suppose a 2-iSAT formula with $n$ variables and $m = cn$ clauses is drawn uniformly at random from the set of all such formulas (with the intervals all in
$[0,1]$).  We estimate the asymptotic probability that such a formula is satisfiable.

\begin{proposition}\label{prop:satisf-of-2sat}
  Let $c' < \nfrac32$.  If $m \le c'n$ then, whp as $n\to\infty$, a randomly drawn 2-iSAT instance is satisfiable.
\end{proposition}
The proof mimics that of Chv\'atal \& Reed~\cite{ChvatalReed92} for the classical 2-SAT very closely; we
include it here just to point out where the number $\nfrac32$ comes in.
\begin{proof}
  Given a fixed bicycle $u_0\rightarrow\dots\rightarrow u_{\ell+1}$ with $r$ clause arcs, the probability
  that it occurs in $G_F$ is at most
  \begin{equation*}
    \lt( \frac{m}{\binom{n}{2}} \rt)^r p^{r-1},
  \end{equation*}
  where $p := \nfrac13$ is the probability that two independently chosen intervals are
  disjoint~\cite{Scheinerman88}.  Hence, the expected number of bicycles with $r$ clause arcs occurring in
  $G_F$ is at most
  \begin{equation*}
    n^{r-1} (r-1)^2 \left( \frac{m}{\binom{n}{2}} \right)^r p^{r-1}
    \;=\;
    \frac{2\frac{m}{n}}{n-1} (r-1)^2 \left(\frac{2pm}{n-1} \right)^{r-1}
    \le
    \frac{2c'}{n-1} (r-1)^2 \left(2pc' \right)^{r-1}.
  \end{equation*}
  where the inequality follows from $\nfrac{m}{n}\le c'$.  Thus, the expected total number of bicycles is
  at most
  \begin{equation*}
    \frac{2 c'}{n-1}\sum_{r=1}^\infty r^2 \left(2pc' \right)^{r-1}.
  \end{equation*}
  The sum is finite if, and only if, $c'<\nfrac32$ for $n\to \infty$. Thus, in this case, the probability
  that a bicycle exists is $o_{c'}(1)$.
\end{proof}

Thus, for every $c' < \nfrac{3}{2}$, whp, a satisfying interpretation can be found by Chepoi et al.'s algorithm~\cite{ChepoiCreignouHermannSalzer10}. We make no
attempt at optimizing this bound as we indeed conjecture that this is the threshold for 2-iSAT.


\newcommand{\gfun}[1]{g_{\!{}_{#1}}}

\section{Total population size of our branching system}\label{sec:Q}

As is done in classical SAT, the sub-routine eliminating the unit clauses can be viewed as a ``discrete time'' queue in which customers (i.e., unit clauses) arrive
per time unit, the number depending on the customer currently serviced, and the single server, corresponding to one run of the inner loop of the algorithm, can
process at least one customer per time unit.  The number of iterations of the sub-routine then roughly corresponds to the length of the (first) busy period of the
server.

Here, since, we are only interested in the length of the first busy period, the ``queue'' is really a branching system, for which we need to know the total number
of individuals which are born before extinction.
Compared to classical SAT, the interval-version poses several small challenges which we address in this section.

Let $a$ be a non-negative integer, and $B(j)$, $j=0,1,2,\dots$, random variables taking values in the non-negative integers.  We say the following sequence of random
variables $Q(j)$ a \textit{discrete queue:}
\begin{align*}
  Q(0) &= 0 \\
  Q(1) &= a \\
  Q(j+1) &=
  \begin{cases}
    a,              &\text{ if } Q(j) = 0 \\
    Q(j)-1+B(j+1)   &\text{ if } Q(j) > 0
  \end{cases}
\end{align*}
The number $Q(j+1)$ is the number of individuals of the branching system after the $j$th individual has reproduced and died.

Denote by $Z$ the length of the first busy period of the server, that is, the total population size of the branching process:
\begin{equation*}
  Z := \sup\{ j\ge 0 \mid Q(i) > 0 \quad \forall i=1,\dots,j \} \quad=\quad \inf \{ j > 0 \mid Q(j) = 0 \} -1.
\end{equation*}

A straightforward adaption of the branching-process based textbook arguments for continuous-time M/G/1-queues gives the following
(see~\ref{apx:Q:qlength-properties}).

\begin{lemma}\label{lem:Q:qlength-properties}
  Suppose the $B(j)$, $j=1,2,\dots$, are iid with mean $\lambda_B$ and common probability generating function $\gfun{B}$.
  The probability generating function $h$ of $Z$ satisfies
  \begin{subequations}
    \begin{equation}\label{eq:Q:probGenF}
      h\Bigl(\tfrac{y}{\gfun{B}(y)}\Bigr) = y^a
    \end{equation}
    for every $y$ for which the power series $\gfun{B}(y)$ converges and does not vanish.  In particular, if $\lambda_B < 1$, we obtain
    \begin{equation}
      \Exp Z = \frac{a}{1 - \lambda_B}.
    \end{equation}
    Moreover, we have
    \begin{equation}\label{eq:Q:exp-moment-method}
      \Prb[ Z \ge \alpha ] \le \frac{\gfun{B}(y)^\alpha}{y^{\alpha-a}}
    \end{equation}
    for all $\alpha > 0$ and $y > 0$ with $y \ge \gfun{B}(y)$.
  \end{subequations}
  \qed
\end{lemma}

\begin{remark}
  Since we are only interested in the first busy period, we make the following modification to the definition of~$Q$: If $Q(j)=0$ but $j > 0$, then we let
  $Q(j+1)=0$ (and not $Q(j+1)=a$ as above).  This makes some inequalities less cumbersome to write down.
\end{remark}

\subsection{Bounding the tail probability for iid binomial $B$}

Let $P$ be a random variable with values in $[0,1]$.  We say that a random variable $B$ has binomial
distribution with random parameter $P$, or $\Bin(m,P)$, if
\begin{equation*}
  \Prb[ B = k \mid P=p ] = \binom{m}{k}p^k(1-p)^{m-k}.
\end{equation*}

In our setting $n$ is a (large) integer, and $m = m(n)$ is an integer depending on $n$.  Define $\lambda = \lambda(n) := \frac{m}{n}$.
Let $P$ be as in~\eqref{eq:Q:def-P}, and suppose that $B$ is $\Bin(m,\nfrac{2P}{n})$.

\begin{lemma}\label{lem:Q:momgen-estimate}
  If $\lambda (y-1) \le \nfrac12$ we have
  \begin{equation*}
    \gfun{B}(y) \le \exp\Bigl( \tfrac{13}{12}\lambda(y-1)  + \tfrac{6}{5}\lambda^2(y-1)^2 \Bigr)
  \end{equation*}
\end{lemma}
\begin{proof}
  We have $e^t \le 1 + t + t^2$ for all $t\le 1$.
  For ease of notation, let $\tau := \Exp P = \nfrac{13}{24}$ and $\tau_2 := \Exp(P^2) = \nfrac{3}{10}$, by Lemma~\ref{lem:rint:moments-of-P}.
  Since $(y-1)\lambda 2 P \le 1$ with probability one, the following estimate holds:
  \begin{multline*}
    \gfun{B}(y)
    = \sum_{k=0}^m \Exp\Prb\bigl( B = k \mid P \bigr)\, y^k
    \\
    = \sum_{k=0}^m \Exp\Bigl(  \tbinom{m}{k}(\tfrac{2P}{n})^k(1-\tfrac{2P}{n})^{m-k} \, y^k \Bigr)
    = \Exp\biggl(\sum_{k=0}^m \tbinom{m}{k}(y\tfrac{2P}{n})^k(1-\tfrac{2P}{n})^{m-k} \biggr)
    \\
    = \Exp\bigl(  (1+(y-1)\tfrac{2P}{n})^m  \bigr)
    \le \Exp\Bigl( e^{2(y-1)\lambda P} \Bigr)
    \le \Exp\lt( 1 + 2(y-1)\lambda P + 4 (y-1)^2\lambda^2 P^2 \rt)
    \\
    = 1 + 2\tau(y-1)\lambda + 4 \tau_2 (y-1)^2\lambda^2
    \le e^{2\tau(y-1)\lambda + 4\tau_2 (y-1)^2\lambda^2}
    = \exp\Bigl( \tfrac{13}{12}\lambda(y-1)  + \tfrac{6}{5}\lambda^2(y-1)^2 \Bigr),
  \end{multline*}
  as claimed.
\end{proof}

Now suppose that $P(j)$, $j=1,2,\dots$, are iid random variables distributed as $P$ defined
in~\eqref{eq:Q:def-P}, and that $B(j)$, $j=1,2,\dots$, are iid random variables distributed as
$\Bin(m,\nfrac{2P(j)}{n})$.

\begin{lemma}\label{lem:Q:iPoi:gen-fun-bound}
  For every $\eps > 0$ there exist $\delta > 0$ and $C \ge 1$ such that, if $\nfrac12 \le
  \frac{13}{12}\lambda \le 1-\eps$, the following is true.

  For all $\alpha \ge C a$, there exists a $y$ with $0 < \gfun{B}(y) < 1 < y \le 2$ such that
  \begin{equation}\label{eq:Q:iPoi:tail}
    \frac{\gfun{B}(y)^\alpha}{y^{\alpha-a}} \le e^{-\delta\alpha}.
  \end{equation}
\end{lemma}
\begin{proof}
  For ease of notation, let $u := y-1$ and $r := \frac{13}{12}\lambda$, so that $\nfrac12 \le r \le 1-\eps$.
  If $0 < u < \frac{1-r}{r} \le 1$, by Lemma~\ref{lem:Q:momgen-estimate}, we may estimate
  \begin{equation*}
    \gfun{B}(y) \le \exp\bigl( \tfrac{13}{12}\lambda u+ \tfrac{6}{5}\lambda^2 u^2 \bigr),
  \end{equation*}
  and thus obtain
  \begin{equation*}
    \Prb[ Z \ge \alpha ] \le \exp\bigl( \alpha(  \tfrac{13}{12}\lambda u + \tfrac{6}{5}\lambda^2 u^2) - (\alpha-a)\log(u+1) \bigr).
  \end{equation*}
  We write the exponent as
  \begin{equation}\label{eq:exponent:oeihfwoqcih}\tag{$*$}
    \alpha r u + \tfrac{6\cdot 12^2}{5\cdot 13^2} \alpha r^2 u^2 - (\alpha-a)\log(u+1).
  \end{equation}
  In order to find a $u$ minimizing~\eqref{eq:exponent:oeihfwoqcih}, we take the derivative and solve the
  resulting quadratic equation
  \begin{equation}\label{eq:innerloop:exponent:osudfbo832hrlsdin}\tag{$**$}
    \tfrac{12^3}{5\cdot 13^2} r^2 u^2
    + \bigl(  r  + \tfrac{12^3}{5\cdot 13^2} r^2  \bigr) u
    - (1-r) + \nfrac{a}{\alpha}  = 0
  \end{equation}
  The value of $u$ which works is the larger one of the two roots:
  \begin{equation}\label{eq:exponent:dsoiewhtroi239ydsh}\tag{$*$$*$$*$}
    u_r :=
       \frac{- \Bigl( 1  + \frac{12^3}{5\cdot 13^2} r  \Bigr)
        +\sqrt{ \Bigl(1  - \frac{12^3}{5\cdot 13^2} r  \Bigr)^2
          + \frac{4\cdot 12^3}{5\cdot 13^2}
        }
      }{\frac{2\cdot 12^3}{5\cdot 13^2} r }
      - O(\nfrac{a}{\alpha}),
  \end{equation}
  with an absolute constant in the $O(\cdot)$ (see \ref{apx:Q:iPoi:busy-prd-tail} for the computation).
  The numerator is greater than zero (implying $y>1$) if, and only if, $4\cdot\frac{12^3}{5\cdot 13^2} r <
  \frac{4\cdot 12^3}{5\cdot 13^2}$, which is equivalent to $r < 1$.  Thus, there exists a $C$ depending
  only on $r$, such that $u_r > 0$ whenever $\alpha \ge C a$.
  Moreover, by letting $u=\frac{1-r}{r}$ in~\eqref{eq:innerloop:exponent:osudfbo832hrlsdin}, we see that $u_r<\frac{1-r}{r} \le 1$, as required.
  Letting $u=u_r$ in~\eqref{eq:exponent:oeihfwoqcih}, we obtain, for $\alpha \ge C a$,
  \begin{equation}\label{eq:Q:iPoi:pre-tail}\tag{$*$$*$$*$$*$}
    \bigl( \delta_r(u_r)+O(\nfrac{1}{C}) \bigr) \alpha,
  \end{equation}
  with an absolute constant in the $O(\cdot)$, where
  \begin{equation*}
    \delta_r(u) = ru + \tfrac{6\cdot 12^2}{5\cdot 13^2} r^2 u^2 - \log(u+1)
  \end{equation*}
  (see \ref{apx:Q:iPoi:busy-prd-tail} for the computation).
  We have $\delta_r(u_r) < 0$, because $\delta_r(0)=0$ and since, by the choice of $u_r$, the derivative of
  $\delta_r$ in the open interval $\lt[0,u_r\rt[$ is negative.  This also implies that $\gfun{B}(y) < 1 < y$.
  Let
  \begin{equation*}
    \delta_* := \max\bigl\{ \delta_r(u_r) \bigm| \nfrac12\le r\le 1-\eps \bigr\} < 0,
  \end{equation*}
  Finally, increase $C$, if necessary, to take care of the dependence on $O(\nfrac{1}{C})$ in~\eqref{eq:exponent:dsoiewhtroi239ydsh}
  and~\eqref{eq:Q:iPoi:pre-tail}, and define $\delta := -\delta_*/2$.
  This completes the proof of the lemma.
\end{proof}

\begin{lemma}\label{lem:Q:binomial:mean-and-tail}
  If $\lambda \le (1-\eps)\frac{12}{13}$, then
  \begin{subequations}
    \begin{align}
      \Exp Z &= \frac{a}{1-\frac{13}{12}\lambda}                                                            \label{eq:Q:iBin:mean}\\
      \intertext{%
        and there exist $\delta > 0$ and $C\ge1$ depending only on $\eps$, such that for all $\alpha \ge C a$ we have the upper tail inequality} %
      \Prb[ Z \ge \alpha ] &\le e^{-\delta\alpha }.  \label{eq:Q:iBin:tail}
    \end{align}
  \end{subequations}
\end{lemma}
\begin{proof}
  Equation~\eqref{eq:Q:iBin:mean} is directly from Lemma~\ref{lem:Q:qlength-properties}.  

  Lemmas \ref{lem:Q:qlength-properties} and~\ref{lem:Q:iPoi:gen-fun-bound} together imply the tail inequality in the case when $\frac{13}{12}\lambda \ge \nfrac12$.
  For smaller values of $\lambda$, we just note that increasing $\lambda$ increases the length of the first busy period, so that the probability for $\lambda :=
  \nfrac{6}{13}$ gives an upper bound for the probability for smaller values of $\lambda$.
\end{proof}

\subsection{Not-independent binomial}
The arrivals at the queue in the context of our algorithm are not completely independent.  Here we deal with the small amount of dependence.

We now describe what kind of $B(j)$ we allow.
The setting is that $n$ is a (large) integer, and that $m = m(n)=\Theta(n)$.
Let $r > 1$ and
\begin{equation}\label{eq:Q:z}
  z = z_r = z_r(n) := \tfrac{r}{\delta} \log n,
\end{equation}
where $\delta$ is as in Lemma~\ref{lem:Q:binomial:mean-and-tail}. 
Suppose that $M(j)$, $N(j)$ are random variables satisfying
\begin{subequations}
  \begin{alignat}{3}
    n-j    &\le N(j) &&\le n+j    &\quad&\text{for all $j$,}\\
    0      &\le M(j) &&\le m      &&     \text{for all $j$,}
    \intertext{with probability one, and }  %
    m^-    &\le M(j) &&\le m^+    &&     \text{for all $j=1,\dots, z$}
  \end{alignat}
  with probability at least $1-O(n^{-r})$.
\end{subequations}
Let the $B(j)$ be distributed as $\Bin( M(j), \frac{P}{N(j)} )$ for all $j$.  More accurately, we assume that there is an iid family of $P(j)$,
$j=1,2,3,\dots$, distributed as $P$ above, and an independent family of random variables $U(j,i)$, $j=1,2,3,\dots$, $i=1,2,3,\dots$ each having uniform
distribution on $[0,1]$, and that the joint distribution of the $B(j)$ is the same as for the family of sums
\begin{equation}\label{eq:Q:dBin-coupling}
  \sum_{i=1}^{M(j)} \Ind\Bigl[ U(j,i) \le \tfrac{P(j)}{N(j)} \Bigr].
\end{equation}
The $P(j)$ and $U(j,i)$ are assumed to be jointly independent, but we make no assumptions about independence regarding the $M(j)$ and $N(j)$ among themselves or
from the $U(j,i)$ and $P(j)$.  However, we do assume that $a$, the $M(j)$, and the $N(j)$ are such that
\begin{equation}\label{eq:Q:ub-total-no-customers}
  a + \sum_{j=1}^\infty B(j) = O(n)
\end{equation}
holds with probability one.

\begin{lemma}\label{lem:Q:final}
  Let $\lambda^\pm = \lambda^\pm(n) := \frac{m^\pm}{n-(\pm z)}$, and suppose $z \ge C a$.  If
  \begin{equation}\label{eq:Q:final:cond-ub-eps-good}
    \lambda^+ \le (1-\eps)\tfrac{12}{13},
  \end{equation}
  then with the $\delta$ and $C$ from Lemma~\ref{lem:Q:binomial:mean-and-tail}, the following holds for large enough $n$:
  \begin{subequations}
    \begin{alignat}{3}
      \frac{a}{1-\frac{13}{12}\lambda^-} - O(n^{1-r}) &\le&\,& \quad\Exp Z
      &&\le \frac{a}{1-\frac{13}{12}\lambda^+}    + O(n^{1-r});                                                 \label{eq:Q:dBin:mean}\\
      \intertext{and for all $\alpha \ge C a$} 
      &&&\Prb[ Z \ge \alpha ] &&\le e^{-\delta\alpha } + O(n^{-r}).                                             \label{eq:Q:dBin:tail}
    \end{alignat}
  \end{subequations}
\end{lemma}
The proof can be found in the appendix: \ref{apx:Q:final}.

\begin{remark}\label{rem:delta-le-1}
  There is no danger in assuming $\delta \le 1$ and $C \ge 1$, and we will do that from this point on.
\end{remark}



\newcommand{\QI}{Q_{\mspace{-1mu}I}}
\newcommand{\QII}{Q_{\mspace{-1mu}I\!I}}
\newcommand{\ZI}{{Z_{\mspace{-1mu}I}}}
\newcommand{\ZII}{Z_{\mspace{-1mu}I\!I}}
\newcommand{\AI}{A_{\mspace{-1mu}I}}
\newcommand{\AII}{A_{\mspace{-1mu}I\!I}}
\newcommand{\BI}{B_{\mspace{-1mu}I}}
\newcommand{\BII}{B_{\mspace{-1mu}I\!I}}

\newcommand{\MI}{M_{\mspace{-1mu}I}}         
\newcommand{\MII}{M_{\mspace{-1mu}I\!I}}

\newcommand{\FI}{F_{\mspace{-1mu}I}}         
\newcommand{\FII}{F_{\mspace{-1mu}I\!I}}

\newcommand{\HI}{H_{\mspace{-1mu}I}}       
\newcommand{\HItwo}{H^{\ge2}_{\mspace{-1mu}I}}
\newcommand{\HII}{H_{\mspace{-1mu}I\!I}}

\newcommand{\DX}{{\Delta\!X}}
\newcommand{\DY}{{\Delta\!Y}}

\section{The inner loop}\label{sec:inner}

Here we analyze Algorithm~\ref{algo:3isat:inner}.  Conditioning on $X(t)$, $Y_2(t)$, and $Y_3(t)$, we
analyze the changes of the parameters $X$, $Y_2$, and $Y_3$ during the $(t+1)$st run of the inner loop, and
bound the probability that an empty clause is generated.

From now on, $n$ and $m$ denote the number of variables and clauses, respectively, in the initial random CNF formula, with $m = cn$ for some constant $c$.  We assume
$c\le 10$, to get rid of some of the letter $c$ in the expressions below.
For any $\eps >0$, we say that $(x,y_2,y_3)\in\RR^3$ is $\eps$-good, if
\begin{equation}
  \begin{aligned}
    \eps n &< x& 
    &\text{ and }&\frac{y_2}{x} &< (1-\eps)\frac{12}{13},
  \end{aligned}  
\end{equation}
and that $\Histt$ is \textit{$\eps$-good} if $(X(t), Y_2(t), Y_3(t))$ is $\eps$-good.


\subsection{Setup of the queues for Phases I and~II}

We now define the queues corresponding to the Phases I and~II.  We will suppress the dependency of the random processes on $\Histt$ in the notation.

We define the queues $\QI$ and $\QII$ for the Phases I and~II, respectively, by modifying Algorithm~\ref{algo:3isat:inner} a little bit.  We will then analyze (the
original) Algorithm~\ref{algo:3isat:inner} with the help of the queues $\QI$ and~$\QII$ defined via this modification.  The changes we make are the following: replace
step~\refinner{algstep:inner:current-var} by
\begin{quote}
  (\refinnernoparen{algstep:inner:current-var}')\ \ If there are unused variables left, choose one uar;
\end{quote}
and step~\refinner{algstep:inner:expose-occ-colored} by
\begin{quote}
  (\refinnernoparen{algstep:inner:expose-occ-colored}')\ \ Expose all occurrences of the current variable $\avar_j$ in clauses colored with a color different from red;
\end{quote}
moreover, in the modification, we do not initiate a repair (since that would kill the queueing process).

Since, with these modifications, red clauses can contain used variables, it is possible to run out of variables before running out of clauses.  It can be easily
verified that this can only happen when all clauses are red.  Hence, in this situation, the modified algorithm will just eat up the red clauses one per iteration.

In the Phase-I queue $\QI$, the number of customers arriving in the first time interval, $\AI$, is the number of red clauses generated by setting the free variable
$\avar_0$ (tentatively) to $\nfrac12$.  Thus, $\AI$ is distributed as $\Bin(Y_2(t),\frac{1}{X(t)})$.  For the iterations $j=1,2,3,\dots$, we find that $\BI(j+1)$ is
the number of uncolored 2-clauses which become red, plus the number of pink 3-clauses which become red, when setting the current variable $\avar_j$ (tentatively) to
$\bar x(\anI_j)$.  Thus, if we denote by $Y'_2(j)$ the number of uncolored 2-clauses plus the number of pink 3-clauses at the beginning of iteration $j$, then
conditioned on $Y'_2(j)$, the distribution of $\BI(j+1)$ is that of $\Bin(Y'_2(j),\frac{P(j+1)}{X(t)-j})$, where as in the previous section, the $P(j+1)$ are iid
random variables distributed as $P$ defined in~\eqref{eq:Q:def-P}.  If we agree on the convention that a $\Bin(0,\nfrac{p}{0})$-variable is deterministically $0$,
this also holds when the queue runs out of variables.

In the Phase-II queue, the number of customers arriving in the first time interval, $\AII$, is the number of unit-clauses generated at the end of Phase~I by setting
the variables to their tentative values.  The $\BII(j)$ are defined analogous to the $\BI(j)$.

At this point, note that the condition~\eqref{eq:Q:ub-total-no-customers}, which is needed for Lemma~\ref{lem:Q:final}, is satisfied for both queues.

\subsection{Bounds for the probabilities of some essential events}

Below, we repeatedly use the following simple Chernoff-type inequality (e.g.~equation~(2.11) in~\cite{JansonLuczakRucinskiBk}): if $U$ is a binomially distributed
random variable with mean $\mu$, then
\begin{equation}\label{eq:simple-chernoff}
  \Prb[ U \ge \alpha ] \le e^{-\alpha}\qquad\text{ for $\alpha \ge 7\mu$}.
\end{equation}

\begin{lemma}\label{lem:inner:X-Y-dont-change-much}
  Let $r > 1$, $1\le z=z(n) = o(n)$ an integer, $(x,y_2,y_3)$ $\eps$-good for some $\eps > 0$, and $m^- := \max(0,y_2 - r z \log n)$, $m^+:= y_2 + r z \log n$.
  For both phases I and~II of the inner loop, the following is true.
  If, at the beginning of the phase at step~\refinner{algstep:inner:start}, there are $x$ variables, $y_2$ 2-clauses, and $y_3$ 3-clauses, then the probability
  that, while dealing with the first $z$ variables in the phase, the number of 2-clauses leaves the interval $[m^-,m^+]$, is $O( n^{-r} )$.
\end{lemma}
\begin{proof}
  For the upper bound $m^+$, the probability that the number of 2-clauses exceeds $m^+$ is bounded from
  above by the probability that one in a sequence of~$z$ independent random variables with
  $\Bin(m,\frac{3}{(\eps n/2)})$-distributions is greater than $r \log n$.  Here the factor $\nfrac12$ on the
  denominator takes care of the $z = o(n)$ variables which are used.  For $n$ large enough, this
  probability, up to a constant factor, is at most
  \begin{multline*}
    z\binom{m}{r\log n} \lt(  \frac{\nfrac{6}{\eps}}{n}  \rt)^{r\log n}
    \le
    z\lt( e\frac{m}{r\log n} \cdot  \frac{\nfrac{6}{\eps}}{n}  \rt)^{r\log n}
    \\
    = 
    z\lt( \frac{e \tfrac{m}{n}\cdot\tfrac{6}{\eps}}{r\log n} \rt)^{r\log n}
    \le 
    z\lt( \nfrac{1}{e^2} \rt)^{r\log n}
    =
    z n^{-2r}
    \le 
    n^{-r}
  \end{multline*}
  For the lower bound $m^-$, the probability can be bounded by the same argument, noting that, if $m^-=0$,
  the corresponding probability is 0.
\end{proof}

Let $R$ denote the event that a repair is invoked during this run of Algorithm~\ref{algo:3isat:inner}.
Moreover, denote by $\ZI$ and $\ZII$ the length of the first busy period of the Phase~I and Phase~II queues, respectively.  Note that they depend on $\AI$ and $\AII$,
respectively.
Further let $\MI$ and $\MII$ be the total number of colored clauses which are generated during Phase~I and Phase~II, respectively; let $\HI$ and $\HII$ the event
that, in some iteration, in steps~\refinner{algstep:inner:expose-occ-colored}, the current variable is found to be contained in a colored clause (other than the
current clause $\aclause_j$); and by $\HItwo$ the probability that in Phase~I the current variable is found to be contained in at least two colored clauses (other
than the current clause $\aclause_j$).

\begin{lemma}\label{lem:inner:probabs-list}
  Suppose that $\Histt$ is $2\eps$-good.
With the $\delta := \delta(\eps)$ and $C := C(\eps)$ from Lemma~\ref{lem:Q:final}, and $r > 1$, the following is true for all $n$ large enough (depending on
  $\eps$).
  \begin{subequations}
    \begin{align}
      \Prb[ \AI \ge r \log n \mid \Histt ] &= O(n^{-r})                                      \label{eq:inner:probabs-list:AI}\\
      \Prb[ \ZI \ge \tfrac{C}{\delta} r\log n  \mid \Histt ] &= O(n^{-r})                    \label{eq:inner:probabs-list:ZI}\\
      \Prb[ \MI \ge \tfrac{500 C}{\eps\delta} r\log n  \mid \Histt ]  &= O(n^{-r})           \label{eq:inner:probabs-list:MI}\\
      \Prb[ \HI \mid \Histt ]  &= O_\eps(\tfrac{\log^2 n}{n})                                      \label{eq:inner:probabs-list:HI}\\
      \Prb[ \HItwo \mid \Histt ]  &= O_\eps(\tfrac{\log^4 n}{n^2})                                      \label{eq:inner:probabs-list:HItwo}\\
      \Prb[ R   \mid \Histt ]  &= O_\eps(\tfrac{\log^2 n}{n})                                      \label{eq:inner:probabs-list:R}\\
      \Prb[ \AII \ge  \tfrac{500 C}{\eps\delta} (r+1)\log n \mid \Histt \aand R] &= O(n^{-r})            \label{eq:inner:probabs-list:AII}\\
      \Prb[ \ZII \ge \tfrac{500 C^2}{\eps\delta} (r+1)\log n \mid \Histt \aand R] &= O(n^{-r})           \label{eq:inner:probabs-list:ZII}\\
      \Prb[ \MII \ge \tfrac{250000 C^2}{\eps^2\delta} (r+1)\log n  \mid \Histt \aand R]  &= O(n^{-r})    \label{eq:inner:probabs-list:MII}\\
      \Prb[ \HII \mid \Histt \aand R]  &= O_\eps(\tfrac{\log^2 n}{n})                                      \label{eq:inner:probabs-list:HII}
    \end{align}
  \end{subequations}
\end{lemma}
\begin{proof}\mbox{}
  \paragraph{\textit{For~\eqref{eq:inner:probabs-list:AI},}}  
  if $\Histt$ is $2\eps$-good, then the probability that $\AI \ge r \log n$ is bounded from above by the probability that a $\Bin(m,\frac{2}{2\eps n})$-variable is
  larger than $r\log n$, which is at most $n^{-r}$, for $n$ large enough, by~\eqref{eq:simple-chernoff}.

  \paragraph{\textit{Proof of~\eqref{eq:inner:probabs-list:ZI}.}}  
  We use Lemma~\ref{lem:Q:final} together with Lemma~\ref{lem:inner:X-Y-dont-change-much} to bound the conditional probability that $\ZI \ge \alpha$.  If $\Histt$
  is $2\eps$-good, then the $m^+$ from Lemma~\ref{lem:inner:X-Y-dont-change-much}, with $x:=X(t)$, $y_2:=Y_2(t)$, $y_3:=Y_3(t)$, and the $z=z_r$
  from~\eqref{eq:Q:z}, is such that~\eqref{eq:Q:final:cond-ub-eps-good} is satisfied if $n$ is large enough depending on $\eps$.

  The requirement for the estimate in~\eqref{eq:Q:dBin:tail} is that $\AI \le a_0 := \min(\nfrac{\alpha}{C},\nfrac{z_r}{C})$.  Thus, for the probabilities
  conditional on $\Histt$, we have
  \begin{multline*}
    \Prb[ \ZI \ge \alpha ]
    \\
    =
    \Prb[ \ZI \ge \alpha \mid \AI \le a_0 ]\Prb[\AI \le a_0]
    + \Prb[ \ZI \ge \alpha \mid \AI > a_0 ]\Prb[\AI > a_0]
    \\
    \le O(e^{-\delta\alpha}) + O(n^{-r}) + \Prb[\AI > a_0].
  \end{multline*}
  With $\alpha := \frac{C}{\delta}r\log n$, using~\eqref{eq:inner:probabs-list:AI} and~\eqref{eq:simple-chernoff}, the right-hand side is $O(n^{-r})$.
  
  \paragraph{\textit{Proof of~\eqref{eq:inner:probabs-list:MI}.}}  
  For every iteration, a clause is only colored if the current variable of the iteration is contained in the clause.  Hence, the number of clauses colored in the
  first $j$ iterations is upper bounded by the sum of $j$ independent $\Bin(m,\frac{3}{\eps n})$-variables.  Hence, the probability that in the first $j$ iterations,
  the number of colored clauses exceeds $j\alpha$ is at most $e^{-\alpha}$ by~\eqref{eq:simple-chernoff}, provided that $\alpha \ge \frac{500}{\eps} j \ge 7\cdot
  \frac{3m}{\eps n/2}j$.  Moreover, we have $\MI \le m$ with probability one.  Thus, conditioning on $\Histt$ (and keeping in mind that $\Histt$ is required to be
  $2\eps$-good), the probability that $\MI$ is larger than $\frac{500 C}{\eps\delta} r\log n$ is at most
  \begin{equation*}
    O(e^{-r\frac{500 C}{\eps\delta} \log n})+ m\Prb[ \ZI \ge r\tfrac{500 C}{\eps \delta} \log n \mid \Histt ]
    = O(n^{-r}) + O(m n^{-500 r})
    = O(n^{-r}).
  \end{equation*}

  \paragraph{\textit{Proofs of~\eqref{eq:inner:probabs-list:HI} and~\eqref{eq:inner:probabs-list:HItwo}.}} 
  In the first phase, in the $j$th iteration, the probability that the current variable $\avar_j$ occurs in a colored clause (other than the current clause
  $\aclause_j$) is $O(\frac{\MI}{X(t)-\ZI})$, and the probability that the number of colored clauses containing $\avar_j$ (other than the current one $\aclause_j$) is
  two or more is $O\bigl( \bigl(\frac{\MI}{X(t)-\ZI}\bigr)^2 \bigr)$.

  By~\eqref{eq:inner:probabs-list:ZI} and~\eqref{eq:inner:probabs-list:MI}, we can bound the probability that this happens in the first $\ZI$ iterations by
  $O_\eps(\frac{\log^2 n}{n}) + O(n^{-r})$ and $O_\eps(\frac{\log^4 n}{n^2}) + O(n^{-r})$, respectively, where the constant in the $O_\eps(\cdot)$ depends only on
  $\eps$.

  \paragraph{\textit{Proof of~\eqref{eq:inner:probabs-list:R}.}} 
  Clearly, the probability that a repair occurs is at most the probability that, in some iteration, the current variable $\avar_j$ occurs in a colored clause
  (other than the current one $\aclause_j$).  Thus, the inequality follow from~\eqref{eq:inner:probabs-list:HI}.

  \paragraph{\textit{Proof of~\eqref{eq:inner:probabs-list:AII}.}} 
  Since $\AII \le \MI$, this inequality follows from \eqref{eq:inner:probabs-list:MI} and \eqref{eq:inner:probabs-list:R}, with $r$ replaced by $r+1$, by conditioning
  on $R$:
  \begin{multline*}
    \Prb[ \MI \ge \tfrac{500 C}{\eps\delta} (r+1)\log n  \mid \Histt \aand R]
    \\
    \le
    \Prb[ \MI \ge \tfrac{500 C}{\eps\delta} (r+1)\log n  \mid \Histt ] / \Prb[ R \mid \Histt ]
    \\
    =
    O(n^{-r-1} \tfrac{n}{\log^2 n})
    = O(n^{-r}).
  \end{multline*}

  \paragraph{\textit{Proof of~\eqref{eq:inner:probabs-list:ZII}.}} 
  We now apply Lemmas \ref{lem:Q:final} and~\ref{lem:inner:X-Y-dont-change-much} to the Phase-II queue.  Let $r':=\frac{500C^2}{\eps\delta}(r+1)$.  If $\Histt$ is
  $2\eps$-good, then the $m^+$ from Lemma~\ref{lem:inner:X-Y-dont-change-much}, with $x:=X(t)$, $y_2:=Y_2(t)$, $y_3:=Y_3(t)$, and the $z=z_{r'}$
  from~\eqref{eq:Q:z}, is such that~\eqref{eq:Q:final:cond-ub-eps-good} is satisfied if $n$ is large enough depending on $\eps$.

  Again, the requirement for the estimate in~\eqref{eq:Q:dBin:tail} is that $\AII \le a'_0 := \min(\nfrac{\alpha}{C},\nfrac{z_{r'}}{C})$.  Thus, for the
  probabilities conditional on $\Histt \aand R$, we have
  \begin{multline*}
    \Prb[ \ZII \ge \alpha ]
    \\
    =
    \Prb[ \ZII \ge \alpha \mid \AII \le a'_0 ]\Prb[\AII \le a'_0]
    + \Prb[ \ZII \ge \alpha \mid \AII > a'_0 ]\Prb[\AII > a'_0]
    \\
    \le O(e^{-\delta\alpha}) + O(n^{-r'}) + \Prb[\AII > a'_0]
  \end{multline*}
  With $\alpha := \tfrac{500 C^2}{\eps\delta} (r+1)\log n$, we have $a'_0 = \tfrac{500 C}{\eps\delta} (r+1)\log n$, so that, by~\eqref{eq:inner:probabs-list:AII}, the
  probability that $\AII > a'_0$ is $O(n^{-r})$.  In total, we obtain an upper bound of $O(n^{-r})$ for the probability that $\ZII \ge \tfrac{500 C^2}{\eps\delta}
  (r+1)\log n$.
  
  \paragraph{\textit{Proof of~\eqref{eq:inner:probabs-list:MII}.}}  
  For every iteration, a clause is only colored if the current variable of the iteration is contained in the clause.  Hence, the number of clauses colored in the
  first $j$ iterations is upper bounded by the sum of $j$ independent $\Bin(m,\frac{3}{\eps n/2}$-variables.  (The factor of $\nfrac12$ in the denominator is to take
  care of the fact that the number of variables, while starting with at least $\eps n$, might drop below $\eps n$ during the run of Phase~I or Phase~II.)  Hence, the
  probability that in the first $j$ iterations, the number of colored variables exceeds $j\alpha$ is at most $e^{-\alpha}$ by~\eqref{eq:simple-chernoff}, provided
  that $\alpha \ge \frac{500}{\eps} j \ge 7\cdot \frac{3m}{\eps n/2}j$.  Moreover, we have $\MII \le m$ with probability one.  Thus, conditioning on $\Histt \aand R$
  (and keeping in mind that $\Histt$ is $2\eps$-good), the probability that $\MII$ is larger than $\frac{500^2 C^2}{\eps^2\delta}(r+1)\log n$ is at most
  \begin{multline*}
    O(e^{-\frac{500^2 C^2}{\eps^2\delta}(r+1)\log n})+ m\Prb[ \ZII \ge \tfrac{500^2 C^2}{\eps^2\delta}(r+1)\log n \mid \Histt ]
    \\
    = O(n^{-r}) + O(m n^{-500 r})
    = O(n^{-r}).
  \end{multline*}
  
  \paragraph{\textit{Proof of~\eqref{eq:inner:probabs-list:HII}.}}  
  In the second phase, in the $j$th iteration, the probability that the current variable $\avar_j$ occurs in a colored clause (other than the current one
  $\aclause_j$) is $O(\frac{\MI}{X(t)-\ZI})$.  By~\eqref{eq:inner:probabs-list:ZII} and~\eqref{eq:inner:probabs-list:MII}, we can bound the probability that this
  happens in the first $\ZII$ iterations by $O_\eps(\frac{\log^2 n}{n}) + O(n^{-r})$, where the constant in the $O_\eps(\cdot)$ depends only on $\eps$.
\end{proof}

\subsection{Changes of the parameters $X(t)$, $Y_2(t)$, and $Y_3(t)$}

We now move to study the differences between successive values of these parameters, and we start with $X(t+1)-X(t)$.
Denote by $\FI$ and $\FII$ the number of iterations of the inner loop in the first and second phase, respectively.
Clearly, $X(t) - X(t+1) = 1 + \FI + \FII$, where the leading $1$ accounts for the free variable $\avar_0$.  Moreover, we have $\FI \le \ZI$ and $\FII \le \ZII$, and
the inequality can be strict for two reasons: in Phase~I, a repair can occur, thus terminating the phase before $\QI$ drops to zero; in both phases a red clause can
vanish (i.e.~become black) in~\refinner{algstep:inner:red1cl}.  However, note that 
\begin{equation}\label{eq:inner:FInoR-equal-ZI}
  \begin{aligned}
    \FI                      &= \ZI                             &&\text{ with probability $1-O_\eps(\tfrac{\log^2 n}{n})$, and} \\
    \FI \Ind[\widebar R ]    &\ge \ZI\Ind[\widebar R ] -1       &&\text{ with probability $1-O_\eps(\tfrac{\log^4 n}{n^2})$}
  \end{aligned}
\end{equation}
by~\eqref{eq:inner:probabs-list:HI}, \eqref{eq:inner:probabs-list:R} and~\eqref{eq:inner:probabs-list:HItwo}.

Let us abbreviate
\begin{equation*}
  \DX := -1 - \frac{ \frac{Y_2(t)}{X(t)} }{ 1- \frac{13Y_2(t)}{12 X(t)} }
  = -1 - \frac{ 12Y_2(t) }{ 12X(t) - 13Y_2(t) }
  = - \frac{ 12X(t) - Y_2(t) }{ 12X(t) - 13Y_2(t) }.
\end{equation*}

\begin{lemma}\label{lem:inner:cond-change-X}
  If $\Histt$ is $2\eps$-good and $n$ large enough depending on $\eps$, then
  \begin{subequations}\label{eq:inner:cond-change-X}
    \begin{align}
      \Bigl| -1-\DX  -  \Exp\bigl( \ZI \bigm| \Histt \bigr)  \Bigr|        &= O_\eps(\tfrac{\log   n}{n})  \label{eq:inner:cond-change-X:ZI}\\
      \Bigl| \DX    -  \Exp\bigl( X(t+1) - X(t) \bigm| \Histt \bigr) \Bigr|&= O_\eps(\tfrac{\log^4 n}{n})  \label{eq:inner:cond-change-X:exp}\\
      \intertext{and}
      \Prb\biggl[ \bigl| X(t+1) - X(t) \bigr| \ge \log^2n \biggm|\Histt\biggr]&= O(n^{-10})                 \label{eq:inner:cond-change-X:tail}
    \end{align}
  \end{subequations}
\end{lemma}
\begin{proof}
  By what we have said above on the relationship between $\FI$, $\FII$ and $X(t+1)-X(t)$, we have $\FI = \ZI\Ind[\widebar R] - E_I$ and $\FII = \ZII - E_{I\!I}$, where
  $E_I$ and $E_{I\!I}$ are error terms accounting for red clauses vanishing.  We have $\Exp(E_I\mid\Histt), \Exp(E_{I\!I}\mid\Histt) = O_\eps(\frac{\log^4 n}{n})$
  by~\eqref{eq:inner:probabs-list:HI} and~\eqref{eq:inner:probabs-list:HItwo} (noting that $E_{I},E_{I\!I} \le m$).

  We compute the mean of $\ZI$ using Lemma~\ref{lem:Q:final} with the $m^\pm$ from
  Lemma~\ref{lem:inner:X-Y-dont-change-much} with $z:=\frac{r}{\delta}\log n$ as in~\eqref{eq:Q:z}.  Thus,
  letting $v := rz\log n$ (the bound from Lemma~\ref{lem:inner:X-Y-dont-change-much}), conditional on $\AI$
  and $\Histt$, we have
  \begin{equation*}
    \frac{\AI}{ 1 - \frac{13 Y_2(t)-v}{12 X(t)+z} } \le \Exp(\ZI\mid\AI\aand\Histt) \le \frac{\AI}{ 1- \frac{13 Y_2(t)+v}{12 X(t)-z} },
  \end{equation*}
  so that 
  \begin{equation*}
    \Exp(\ZI\mid\AI\aand\Histt) = \frac{\AI}{1-\frac{13 Y_2(t)}{12 X(t)} } + O_\eps(\tfrac{\AI \log n}{n}),
  \end{equation*}
  provided that $\AI \le \nfrac{z}{C}$, which holds with probability at least $1-O(n^{-2})$ by~\eqref{eq:inner:probabs-list:AI} by increasing, if necessary, $r$
  beyond~$2\delta C$.  Since $\ZI = O(n)$ with probability one, we obtain
  \begin{multline*}
    \Exp(\ZI\mid\Histt)
    = \Exp\biggl(\frac{\AI}{ 1- \frac{13 Y_2(t)}{12 X(t)} }  +O_\eps(\tfrac{\AI \log n}{n})  \biggm|\Histt\biggr) 
    \\
    = \frac{\Exp(\AI\mid\Histt)}{ 1- \frac{13 Y_2(t)}{12 X(t)} } +O_\eps(\tfrac{\log n}{n})
    = \frac{\frac{Y_2(t)}{X(t)}}{ 1- \frac{13 Y_2(t)}{12 X(t)} } +O_\eps(\tfrac{\log n}{n}),
  \end{multline*}
  which proves~\eqref{eq:inner:cond-change-X:ZI}.
  For $\FI$, we obtain
  \begin{multline*}
    \Exp(\FI\mid \Histt) = \Exp(\ZI\mid\Histt) - \Exp(\ZI\Ind(R)\mid\Histt) - \Exp(E_I\mid\Histt)
    \\
    = -1 - \DX + O_\eps(\tfrac{\log   n}{n}) - O_\eps(\log n)\Prb(R\mid\Histt) - mO(n^{-r}) - O_\eps(\tfrac{\log^4 n}{n})
    \\
    = -1 - \DX + O_\eps(\tfrac{\log^4 n}{n})
  \end{multline*}
  and
  \begin{equation*}
    \Exp(\FII\mid \Histt \aand R) \le \Exp(\ZII) = O_{\eps}(\log n) + O(n^{-r})m,
  \end{equation*}
  from which~\eqref{eq:inner:cond-change-X:exp} follows.

  Since $X(t) - X(t+1) \le 1 + \ZI + \ZII$, the tail inequality~\eqref{eq:inner:cond-change-X:tail} follows immediately from~\eqref{eq:inner:probabs-list:ZI}
  and~\eqref{eq:inner:probabs-list:ZII}.
\end{proof}

\begin{lemma}\label{lem:inner:cond-change-Y3}
  If $\Histt$ is $2\eps$-good, then
  \begin{subequations}
    \begin{align}
      \Biggl| \DX \frac{3Y_3(t)}{X(t)}  - \Exp\bigl( Y_3(t+1) - Y_3(t) \bigm| \Histt \bigr) \Biggr|&=O_\eps(\tfrac{\log^4 n}{n}) \label{eq:inner:cond-change-Y3:exp}\\
      \intertext{and}
      \Prb\biggl[ \bigl| Y_3(t+1) - Y_3(t) \bigr| \ge \log^2 n\biggm|\Histt\biggr]                  &=O(n^{-10})                  \label{eq:inner:cond-change-Y3:tail}
    \end{align}
  \end{subequations}
\end{lemma}
\begin{proof}
  Let us denote by $X'(j)$ the number of unused variables after $j$ iterations of the inner loop, i.e., before $\avar_j$ is used, for $j=0,1,2,\dots$.
  In every iteration of the inner loop, regardless of whether in Phase~I or Phase~II, for every uncolored 3-clause $\aclause$, there is a $\frac{3}{X'(j)}$
  probability that the current variable $\avar_j$ is found to be contained in $\aclause$ in step~(\refinnernoparen{algstep:inner:hit-remaining-occ}),
  or~(\refinnernoparen{algstep:inner:AI}.3), respectively, for the zeroth iteration in Phase~I.  If that is the case, the 3-clause is colored, and
  when the inner loop terminates, the clause will no longer be a 3-clause.

  If we suppose that, at the beginning of iteration $j=0,1,2,\dots$, before the current variable $\avar_j$ is treated, there are $Y'_3(j)$ uncolored 3-clauses and
  $X'(j)$ unused variables, then the number of 3-clauses which are hit by $\avar_j$ is distributed as $\Bin(Y'_3(j),\nfrac{3}{X'(j)})$.  (We have $X'(j) = X(t) - j$ in
  Phase~I, but in Phase~II the value of course depends on how Phase~I went.)
  
  For~\eqref{eq:inner:cond-change-Y3:tail}, we can just use the fact that the number of 3-clauses which are colored is bounded from above by $\MI + \MII$, the total
  number of colored clauses.  Thus, by~\eqref{eq:inner:probabs-list:MI} and~\eqref{eq:inner:probabs-list:MII}, this number is at most $\log^2 n$ with probability
  $1-O(n^{-10})$ for $n$ large enough depending on $\eps$.

  For the conditional expectation estimate~\eqref{eq:inner:cond-change-Y3:exp}, we compute, conditional on $\Histt$,
  \begin{equation*}
    \Exp\bigl( Y_3(t+1)-Y_3(t) \bigr)
    = \Exp\bigl( (Y_3(t+1)-Y_3(t))\Ind[R] \bigr) + \Exp\bigl( (Y_3(t+1)-Y_3(t))\Ind[\widebar R] \bigr) .
  \end{equation*}
  For the left summand, we have
  \begin{multline*}
    \Exp\bigl( (Y_3(t+1)-Y_3(t))\Ind[R] \bigr)
    \\
    \le 
    \Exp\bigl( \log^2 n \Ind[R\aand Y_3(t+1)-Y_3(t) \le \log^2 n] \bigr)
    \\
    \shoveright{+ \Exp\bigl( m \Ind[R\aand Y_3(t+1)-Y_3(t) \ge \log^2 n] \bigr)}
    \\
    \le 
    \log^2 n\; \Prb[R] + m \Prb[Y_3(t+1)-Y_3(t) \ge \log^2 n]
    = \log^2 n\; O_\eps(\tfrac{\log^2 n}{n}) + O(n^{-9})
    \\
    = O_\eps(\tfrac{\log^4 n}{n}),
  \end{multline*}
  by~\eqref{eq:inner:probabs-list:R} and~\eqref{eq:inner:cond-change-Y3:tail}.

  For the right summand, we have
  \begin{equation*}
    \Exp\bigl( (Y_3(t+1)-Y_3(t))   \Ind[\widebar R] \bigr)
     = \Exp\biggl( \sum_{j=1}^{\ZI+1} G(j) \Ind[\widebar R] \biggr)
     + O_\eps(\tfrac{\log^2 n}{n}),
  \end{equation*}
  where, conditioned on $Y'_3(j)$ as defined above, the $G(j+1)$ are distributed as $\Bin(Y'_3(j),\frac{3}{X(t)-j})$, and the $O(\cdot)$ accounts for the possibility
  that $\FI < \ZI$, cf.~\eqref{eq:inner:FInoR-equal-ZI}.
  Using~\eqref{eq:simple-chernoff} and a similar argument as above, we see that
  \begin{equation*}
    \Exp\biggl( \sum_{j=1}^{\ZI+1} G(j) \Ind[\widebar R] \biggr)
    = \Exp\biggl( \sum_{j=1}^{\ZI+1} G(j) \biggr) + O_\eps(\tfrac{\log^4 n}{n}).
  \end{equation*}

  Computing the expectation of the sum can be done in the same way as for classical SAT (e.g.\ in~\cite{Achlioptas00,AchlioptasSorkin00,Achlioptas01}).  Indeed, using
  the optional stopping theorem ($\ZI+1$ is a stopping time for the history of the queue together with all random processes involved; cf.~the proof of the next lemma
  for the details, where the situation is essentially the same, only a bit more complicated), we find that
  \begin{equation*}
    \Exp\biggl( \sum_{j=1}^{\ZI+1} G(j) \biggr) = \Exp\biggl( \sum_{j=0}^{\ZI} \frac{3Y'_3(j)}{X(t)-j} \biggr),
  \end{equation*}
  where we agree that $0/0 = 0$.
  By~\eqref{eq:inner:cond-change-Y3:tail}, $Y_3(t)-\log^2 n \le Y'_3(j) \le Y_3(t)$ with probability $1-O(n^{-10})$, and by~\eqref{eq:inner:probabs-list:ZI} we have
  $\ZI \le \log^2 n$, implying $X(t)-j \ge \frac12 X(t)$, with probability $1-O(n^{-10})$.
  Thus, we conclude
  \begin{multline*}
    \Exp\biggl( \sum_{j=0}^{\ZI} \frac{3 Y'_3(j)}{X(t)-j} \biggr)
    \\
    = \Exp\Biggl( \Ind\bigl[ Y_3(t)-\log^2 n \le Y'_3(j) \;\aand\; X(t)-j\ge \tfrac12 X(t) \bigr]
    \cdot \sum_{j=0}^{\ZI} \biggl( \frac{3 Y_3(t)}{X(t)} + O\Bigl(\frac{X(t) \log^2 n}{X(t)^2}\Bigr)\biggr) \Biggr) \\
    \shoveright{+O(n^{-7})}
    \\
    = \bigl( 1 + \Exp\ZI \bigr) \biggl( \frac{3 Y_3(t)}{X(t)} + O\Bigl(\frac{\log^2 n}{n}\Bigr) + O(n^{-7})) \biggr)
    = -\DX \frac{3 Y_3(t)}{X(t)} + O\Bigl(\frac{\log^2 n}{n} \Bigr),
  \end{multline*}
  by~\eqref{eq:inner:cond-change-X:ZI}.  This concludes the proof of~\eqref{eq:inner:cond-change-Y3:exp}.
\end{proof}

\newcommand{\sto}{\shortrightarrow}
\begin{lemma}\label{lem:inner:cond-change-Y2}
  If $\Histt$ is $2\eps$-good, then
  \begin{subequations}
    \begin{align}
      \biggl| \frac{3 Y_3(t)}{2 X(t)}  -  (\DX + 1)\frac{13 Y_3(t)}{8 X(t)} + \DX \frac{2 Y_2(t)}{X(t)} 
      -\Exp\bigl( Y_2(t+1) - Y_2(t) \bigm| \Histt \bigr) \biggr|                  & = O_\eps(\tfrac{\log^4 n}{n})       \label{eq:inner:cond-change-Y2:exp}\\
      \intertext{and}
      \Prb\biggl[ \bigl| Y_2(t+1) - Y_2(t) \bigr| \ge \log^2 n\biggm|\Histt\biggr] & = O(n^{-10})                        \label{eq:inner:cond-change-Y2:tail}
    \end{align}
  \end{subequations}
\end{lemma}
\begin{proof}
  The tail inequality is obtained by referring to~\eqref{eq:inner:probabs-list:MI} and~\eqref{eq:inner:probabs-list:MII} again, since very clause which changes its
  length has been colored before that can happen.

  Let us denote by $X'(j)$ the number of unused variables after $j$ iterations of the inner loop, i.e., before $\avar_j$ is selected.
  In every iteration of the inner loop, regardless of whether in Phase~I or Phase~II, for every uncolored 2-clause $\aclause$, there is a $\frac{2}{X'(j)}$
  probability that the current variable $\avar_j$ is found to be contained in $\aclause$ in step~(\refinnernoparen{algstep:inner:while}.\ref{algstep:inner:hit-remaining-occ}),
  or~(\refinnernoparen{algstep:inner:AI}.3), respectively, for the zeroth iteration in Phase~I.  If that is the case, the 2-clause is colored, and
  when the inner loop terminates, the clause will no longer be a 2-clause.  The same is true for 3-clauses which have become red in some previous iteration.
  Denote the total number of 2-clauses and pink 3-clauses which are hit by the current variable in some iteration over the whole run of
  Algorithm~\ref{algo:3isat:inner} by $L_{2\times}$.
  
  The analysis of the expectation and tail of $L_{2\times}$ is almost identical to the analysis done in the previous lemma for the 3-clauses.  Here, too, we have to
  condition on the number of uncolored 2-clauses and pink 3-clauses not changing too much.  The difference is the need to control the number of pink 3-clauses and,
  after a repair, the number of 3-clauses becoming 2-clauses.  The latter two numbers are bounded from above by $Y_3(t+1)-Y_3(t)$, which is at most $\log^2 n$ with
  probability $1-O(n^{-10})$.  Thus, for $L_{2\times}$, we just note that its expectation accounts for the summand $- \DX \frac{2 Y_2(t)}{X(t)}$
  in~\eqref{eq:inner:cond-change-Y2:exp}.
  
  Now let us denote the number of 3-clauses which become 2-clauses during the two phases of the inner loop by $L_{3\sto 2}$, and let us also focus on the case when no
  repair occurs.

  In this case $L_{3\sto2}$ behaves similarly to $Y_3(t+1)-Y_3(t)$, with two differences: The probabilities that a 3-clause is colored pink is different; and the
  probability in the 
  zeroth iteration differs from the others.
  Let us first consider the zeroth iteration.  The probability that the tentative value $\nfrac12$ of $\avar_0$ makes a 3-clause pink is $\nfrac12$ by
  Lemma~\ref{lem:rint:x-in-I}.  Thus, if there is no repair, this contribution is distributed as $\Bin(Y_3(t),\frac12\cdot\frac{3}{X(t)})$.

  For the other iterations, $j=1,2,3,\dots$, if an uncolored 3-clause $\aclause$ contains the current variable $\avar_j$, the probability that $\aclause$ becomes pink
  in~\refinner{algstep:inner:colorize} depends on the current interval $\anI_j$, and is distributed as~$P$ defined in~\eqref{eq:Q:def-P}.  Indeed, if we denote the
  number of uncolored 3-clauses in iteration $j$ by $Y'_3(j)$ again, then, conditioned on $Y'_3(j)$ and $X'(j)$, the number $G(j+1)$ of uncolored 3-clauses which
  become pink in iteration $j$ is distributed as $\Bin(Y'_3(j),\frac{3P(j+1)}{X'(t)})$, i.e., binomial with random parameter $P(j+1)$.  The $P(j)$ are the iid random
  variables distributed as $P$ in~\eqref{eq:Q:def-P} defined by $\bar x(\anI_j)$, in other words $P(j+1) = 1-2\bar x(\anI_j)(1-\bar x(\anI_j))$.

  Let $G(1)$ be distributed as $\Bin(Y_3(t),\frac{3}{2X(t)})$, define $D(j+1) := G(j+1) - \frac{13 Y'_3(j)}{8 X'(j)}$, where we agree that $0/0 = 0$, and denote by
  $\mathscr F(j)$ the history of the process up to iteration $j$, i.e., before the variable $\avar_j$ is treated.  Then $\sum_{j=1}^\ell D(j)$, $\ell=1,2,3,\dots$, is
  a martingale with respect to $\mathscr F(j)$, $j=0,1,2,\dots$, and $\ZI+1$ is a stopping time, because deciding whether $\ZI + 1 \le \ell$ amounts to checking
  whether $\QI(\ell) = 0$.

  To estimate the expectation of the contribution of these, we use the optional stopping theorem again; note that the stopping time is finite with probability one,
  because $\ZI \le m$.  We conclude that $\Exp\lt(\sum_{j=1}^{\ZI+1} D(j)\rt) = 0$, which means
  \begin{equation*}
    \Exp\Biggl(\sum_{j=1}^{\ZI+1} G(j)\Biggr) = \Exp\Biggl(\sum_{j=0}^{\ZI} \frac{13 Y'_3(j)}{8 X'(j)} \Biggr).
  \end{equation*}
  Arguing as we have done a number of times in regard of the possible deviations of $Y'(j)$ from $Y(t)$, we see that the right hand side equals
  \begin{equation*}
    \bigl( \Exp\ZI +1 \bigr) \frac{13 Y_3(t)}{8 X(t)} +O_\eps\bigl(\tfrac{\log^4 n}{n}\bigr).
  \end{equation*}

  Getting rid of the conditioning on the event that no repair occurs is done in the same way as in the previous lemma, and we leave the details to the reader.
\end{proof}

\subsection{Failure probability}

We now bound the probability that an empty clause is generated by a run of the inner loop, including, possibly, the repair and following second phase.

\begin{lemma}\label{lem:inner:no-failure}
  If $\Histt$ is $2\eps$-good, then the probability that Algorithm~\ref{algo:3isat:inner} produces an empty clause, is $o(\nfrac{1}{n})$.
\end{lemma}
\begin{proof}
  We use Lemma~\ref{lem:3isat:failure}.
  Let us first deal with Phase~II.
  The probability that the algorithm \hurzes\ in Phase~II is $O_\eps(\frac{\log^2n}{n})$ by~\eqref{eq:inner:probabs-list:HII}, conditioned on a repair
  occurring, so that by the law of total probability, the probability that the algorithm \hurzes\ in Phase~II is at most $O_\eps(\frac{\log^4n}{n^2})$,
  by~\eqref{eq:inner:probabs-list:R}.
  
  For Phase~I, we need to go through the possible reasons for the algorithm to \hurz.  First of all, by~\eqref{eq:inner:probabs-list:HItwo}, the
  probability that the current variable $\avar_j$ is contained in a colored (red or not) clause other than the current one $\aclause_j$ is
  $O_\eps(\frac{\log^4n}{n^2})$, which takes care of step~(\refinnernoparen{algstep:inner:while}.\ref{algstep:inner:double-hit:obacht}).
  
  The probability that a fixed clause contains the current variable of a fixed iteration depends only on the number of variables and the number of unexposed atoms
  in the clause, and so it can always be bounded by $\frac{3}{\eps n}$.
  In order for a 3-clause to become red or blue (or even black), it must contain the current variable of (at least) two iterations.  The probability of this
  happening is $O_\eps(\frac{\log n}{n^2})$, where we have used~\eqref{eq:inner:probabs-list:ZI}.  This gives the case of
  step~(\refinnernoparen{algstep:inner:red1cl}.\ref{algstep:inner:hyperedge:obacht}).
  
  Similarly, for step~(\refinnernoparen{algstep:inner:red1cl}.\ref{algstep:inner:cycle:obacht}), a 2-clause must have been hit twice by the current variable of an
  iteration, the probability of which is again bounded by $O_\eps(\frac{\log n}{n^2})$.
  
  In total, the failure probability can be bounded by $O(\frac{\polylog n}{n^2})$
\end{proof}



\section{The outer loop}\label{sec:outer}

At the heart of analysis of the outer loop is the well-known theorem of Wormald's which, in certain situations, allows to estimate parameters of random processes
by solutions to differential equations.  Here is the first goal of our analysis.

\begin{lemma}\label{lem:3isat:ivp-has-sol}
  For every $c \in \lt]0,3\rt]$, the initial value problem
  \begin{subequations}\label{eq:3isat:ivp-dx}
  \begin{align}
    \frac{dy}{dx} &= \frac{-18cx^4 + 2y(12x-y)}{x(12x - y)}                                                                          \label{eq:3isat:ivp-dx:ode}\\
    y(1) &= 0                                                                                                                        \label{eq:3isat:ivp-dx:inival}
  \end{align}    
  \end{subequations}
  has a unique solution $y$ defined on the interval $\lt]0,1\rt]$.
\end{lemma}

See Fig~\ref{fig:3isat:ivp} for a rough sketch of the direction field~\eqref{eq:3isat:ivp-dx:ode} with $c=2$, and a solution to the IVP.  Since, ultimately, we
will solve the IVP~\eqref{eq:3isat:ivp-dx} numerically for the right value of $c$ anyway, strictly speaking, this lemma is not needed to complete our argument.
However, we would like to reduce our reliance on numerical computations as much as possible.

\begin{figure}[htp]
  \centering
  \includegraphics[scale=.3333333333333333333333]{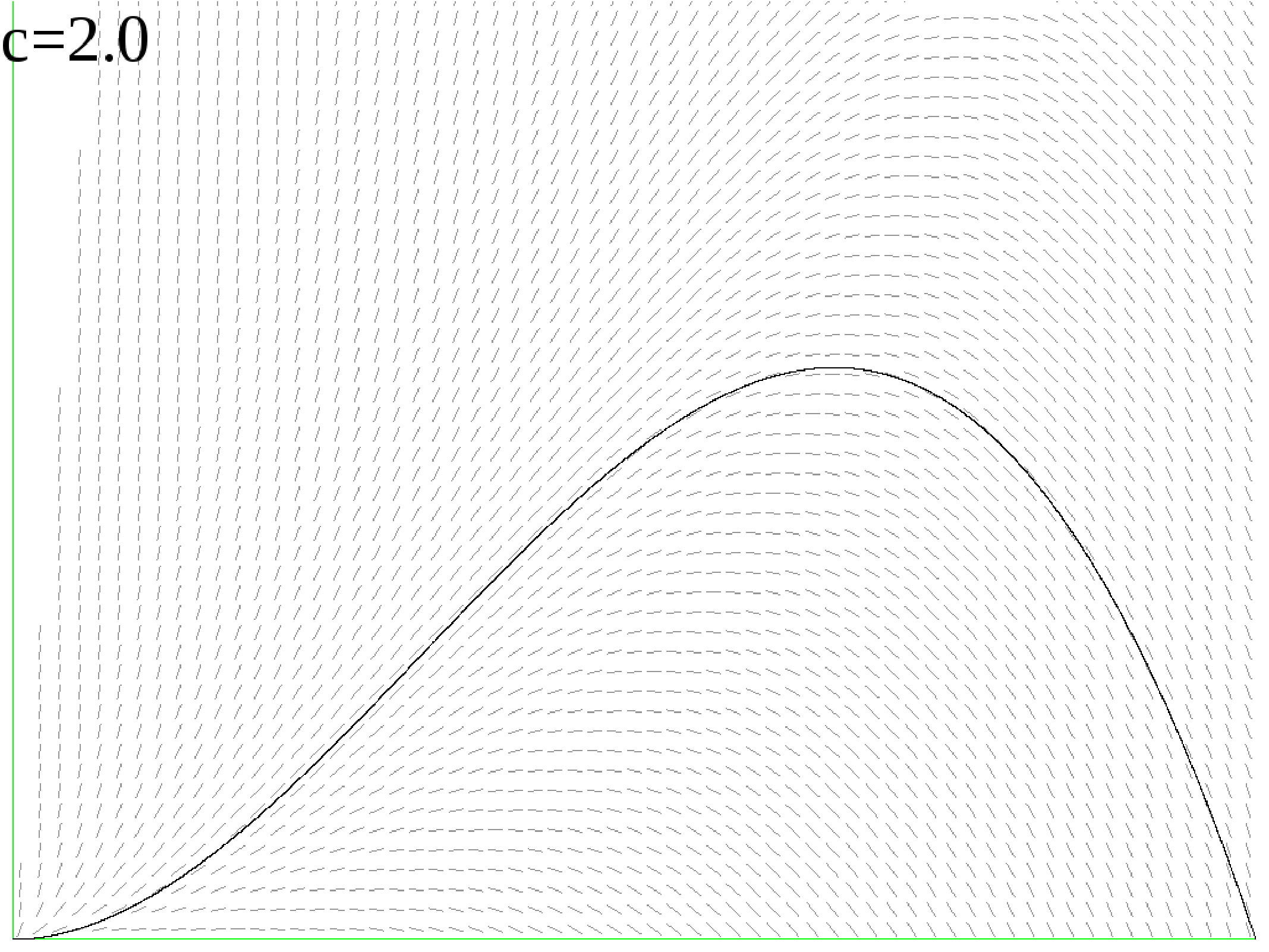}
  \caption{Direction field and solution for IVP~\eqref{eq:3isat:ivp-dx}}
  \label{fig:3isat:ivp}
\end{figure}

\begin{proof}[Proof of Lemma~\ref{lem:3isat:ivp-has-sol}]
  To use the known theorems on IVPs, note that the right hand side of~\eqref{eq:3isat:ivp-dx:ode}, seen as a function of $(x,y)$, is continuously differentiable on
  $\{ (x,y) \mid x > 0, y < 12 x \}$.

  We make the following claims:
  \begin{enumerate}[\it(a)]
  \item\it\label{enum:3isat:ivp-has-sol:right} For $\nfrac45 \le x \le 1$, the solution to the IVP never crosses the line $y=5(1-x)$;
  \item\it\label{enum:3isat:ivp-has-sol:left} for $0 < x \le \nfrac45$, the solution to the IVP never crosses the line $y=6x$.
  \end{enumerate}
  Thus, the solution to the IVP does not approach the line $y=12x$, which implies that the solution extends to the whole interval $\lt]0,1\rt]$.

  Let $g(x,y) := \frac{-18cx^4 + 2y(12x-y)}{x(12x - y)}$, the right hand side of the ODE~\eqref{eq:3isat:ivp-dx:ode}.
  To prove claim~(\ref{enum:3isat:ivp-has-sol:left}), it suffices to show that, with $y(x) := 6x$, whenever $0<x\le\nfrac45$, we have $\frac{dy}{dx} < g(x,y(x))$.
  The computation is easy but tedious and can be found in the appendix, see~\ref{apdx:3isat:ivp-has-sol}.
  Similarly, for claim~(\ref{enum:3isat:ivp-has-sol:right}), with $y(x) := 5(1-x)$, for every $\nfrac45 \le x\le 1$, we have $\frac{dy}{dx} < g(x,y(x))$.  The
  computation is in the appendix, too.
\end{proof}


\begin{lemma}\label{lem:outer:params-stickto-IVP}
  Let $c\le 3$ and $y$ a solution to~\eqref{eq:3isat:ivp-dx}, and let $x_0$ be the infimum over all $x \ge 3\eps$ for which
  \begin{equation}\label{eq:3isat:critic-line}
    13 y(x) < (1-3\eps) 12 x
  \end{equation}
  holds.
  Then there exists a $\tau >0$ and a strictly decreasing smooth function $x\colon[0,\tau]\to\RR$ with $x(0) = 1$ and $x(\tau)=x_0$, such that whp for all~$t$ with
  $\nfrac{t}{n} < \tau$:
  \begin{subequations}%
    \begin{align}
      X(t)   &= n\, x(\nfrac{t}{n}) + o(n) \\
      Y_2(t) &= n\, y(x(\nfrac{t}{n})) + o(n)         \\
      Y_3(t) &= n\, c x(\nfrac{t}{n})^3 +o(n).
    \end{align}
    Moreover, we have the relationship
    \begin{equation}\label{eq:outer:ode-dt:dx}
      \frac{dx}{dt}  = -1 - \frac{ y(x) }{ x - \frac{13}{12} y(x) } = -\frac{ 12\,x - y(x) }{ 12\,x - 13\,y(x) }
    \end{equation}
  \end{subequations}
\end{lemma}
\begin{proof}
  For the proof we use Wormald's well-known theorem, which requires some set up and computations.  Using the notation of Theorem~5.1 in \cite{Wormald-LecNot-99},
  let
  \begin{align*}
    D         &:= \bigl\{ (t,x,y_2,y_3)\in \lt]-\eps,c+\eps\rt[^{\,4} \bigm| \text{$(nx,ny_2,ny_3)$ is $2\eps$-good} \bigr\}\\
    C_0       &:= 10\\
    \beta     &:= \log^2 n \\
    \gamma    &:= 3n^{-2} \\
    \lambda_1 &:= \frac{\log^5 n}{n} \\
    \lambda   &:= \frac{\log^{\nfrac73} n}{n^{\nfrac13}}, \\
  \end{align*}
  Note that $\lambda > \lambda_1 + C_0 n\gamma$, and $\lambda = o(1)$, as required in Theorem~5.1 in \cite{Wormald-LecNot-99}.

  Obviously, we have $0 \le X, Y_2, Y_3 < C_0 n$.
  \begin{enumerate}[(i)]
  \item Equations~\eqref{eq:inner:cond-change-X:tail},~\eqref{eq:inner:cond-change-Y2:tail}, and~\eqref{eq:inner:cond-change-Y3:tail}, respectively, show that, if
    $(t/n,X(t)/n,Y_2(t)/n,Y_3(t)/n)\in D$, then, conditioned on $\Histt$, the probability that $X(t+1)-X(t) \le \beta$, $Y_2(t+1)-Y_2(t) \le \beta$, and
    $Y_3(t+1)-Y_3(t) \le \beta$ hold, is at least $1-\gamma$.
  \item The first parts of Lemmas~\ref{lem:inner:cond-change-X},~\ref{lem:inner:cond-change-Y2}, and~\ref{lem:inner:cond-change-Y3}, respectively, show that, if
    $(t,x,y_2,y_3) := (t/n,X(t)/n,Y_2(t)/n,Y_3(t)/n)\in D$,
    \begin{align*}
      \biggl| f(t,x,y_2,y_3)   - \Exp\bigl( X(t+1) - X(t) \bigm| \Histt \bigr) \biggr| &\le \lambda_1\\
      \biggl| g_2(t,x,y_2,y_3) - \Exp\bigl( Y_2(t+1) - Y_2(t) \bigm| \Histt \bigr) \biggr| &\le \lambda_1\\
      \biggl| g_3(t,x,y_2,y_3) - \Exp\bigl( Y_3(t+1) - Y_3(t) \bigm| \Histt \bigr) \biggr| &\le \lambda_1,
    \end{align*}
    where
    \begin{align*}
      f(t,x,y_2,y_3)   &:= -1 - \frac{ 12y_2(t) }{ 12x(t) - 13y_2(t) }\\
      g_2(t,x,y_2,y_3) &:= \frac{3 y_3(t)}{2 x(t)}  +  (-1-f(t,x,y_2,y_3))\frac{13 y_3(t)}{8 x(t)} + f(t,x,y_2,y_3) \frac{2 y_2(t)}{x(t)} \\
      g_3(t,x,y_2,y_3) &:= f(t,x,y_2,y_3) \frac{3y_3(t)}{x(t)}.
    \end{align*}
  \item There exists an $L$ depending on $\eps$ such that $f,g_2,g_3$ are $L$-lipschitz continuous on $D$.
  \end{enumerate}
  Let $x,y_2,y_3$ be the solution to the initial value problem
  \begin{subequations}\label{eq:outer:ivp-dt}
    \begin{align}
      \frac{dx}{dt}   &= f(t,x(t),y_2(t),y_3(t))                    \label{eq:outer:ivp-dt:dx}\\
      \frac{dy_2}{dt} &= g_2(t,x(t),y_2(t),y_3(t))                \label{eq:outer:ivp-dt:dy2}\\
      \frac{dy_3}{dt} &= g_3(t,x(t),y_2(t),y_3(t))                \label{eq:outer:ivp-dt:dy3}
    \end{align}
    \begin{align}
      x(0)          &= 1                       &     y_2(0)          &= 0                           &  y_3(0)          &= c.      
    \end{align}
  \end{subequations}
  From Wormald's theorem, we conclude that with probability
  \begin{equation*}
    1 - O\Bigl(  n\gamma \tfrac{\beta}{\lambda}e^{-n (\nfrac{\lambda}{\beta})^3}  \Bigr) = 1 - O(\tfrac1n),
  \end{equation*}
  it is true that, for all $t=0,\dots,\sigma n$, we have $X(t) = nx(t/n) + O(\lambda n)$, $Y_2(t) = ny_2(t/n) + O(\lambda n)$, and $Y_3(t) = ny_3(t/n) +
  O(\lambda n)$, where $\sigma=\sigma(n)$ is the supremum over all $s$ for which the solution to~\eqref{eq:outer:ivp-dt} can be extended before reaching within a
  distance of $C\lambda$ from the boundary of $D$, for a large constant $C$.
  
  We now need to study the initial value problem~\eqref{eq:outer:ivp-dt}.
  Let us start with the first equation~\eqref{eq:outer:ivp-dt:dx},  which we write as
  \begin{equation*}
    \frac{dx}{dt} = - \frac{ 12\,x - y_2 }{ 12\,x - 13\,y_2 },
  \end{equation*}
  which amounts to
  \begin{equation}\label{eq:outer:ode-sepsol-dt}
    -dt = \frac{ 12\,x - 13\,y_2 }{ 12\,x - y_2 }
    \,dx
    =  \Bigl(  1 - \frac{12\,y_2}{ 12\,x - y_2 } \Bigr)\,dx,
  \end{equation}
  The third inequality
  \begin{equation*}
    \frac{dy_3}{dt} = \frac{dx}{dt} \frac{ 3 y_3 }{x},
  \end{equation*}
  is equivalent to
  \begin{equation*}
    \frac{dy_3}{dx} = \frac{3 y_3}{x},
  \end{equation*}
  which immediately integrates to\footnote{%
    It should be noted that this is the same relationship between $x$ and $y_3$ as in the case of classical 3-SAT (see~\cite{Achlioptas01}).  %
  }%
  \begin{equation*}
    y_3 = c x^3,
  \end{equation*}
  where the constant before the $x^3$ is derived from the initial value conditions $y_3(0) = c$ and $x(0) = 1$.
  Finally, we write the second equation as
  \begin{equation*}
    \frac{dy_2}{dt} %
    = -\frac{y_3}{8 x} - \frac{dx}{dt} \; \frac{13 y_3}{8 x} + \frac{dx}{dt} \; \frac{2 y_2}{x}
    = -\frac{y_3}{8 x} - \frac{13}{8} cx^2\frac{dx}{dt} + \frac{2 y_2}{x}\frac{dx}{dt}
  \end{equation*}
  from which we obtain
  \begin{equation*}
    \frac{dy_2}{dx} =  -\frac{c}{8}x^2 \frac{dt}{dx} - \frac{13}{8} cx^2 + \frac{2 y_2}{x},
  \end{equation*}
  which, by~\eqref{eq:outer:ode-sepsol-dt}, yields
  \begin{equation*}
    \frac{dy_2}{dx}
    = \frac{c}{8}x^2 \frac{ 12\,x - 13\,y_2 }{ 12\,x - y_2 }   -   \frac{13}{8} cx^2   +   \frac{2 y_2}{x}
    = \frac{-18cx^4 + 2y_2(12x-y_2)}{x(12x - y_2)}
  \end{equation*}
  which is an ODE of the function $y_2$ in the variable $x$.  In fact, with $y_2(1)=0$, we recognize the IVP~\eqref{eq:3isat:ivp-dx}, and thus $y=y_2$ in the
  interval on which both are defined.

  To summarize, we have $y_3 = cx^3$, and $y_2=y$ as a function of $x$ is a solution to the IVP~\eqref{eq:3isat:ivp-dx}, and $x$ as a function of $t$ solves the
  ODE~\eqref{eq:outer:ivp-dt:dx} with boundary condition $x(0) = 1$.
  

  From Lemma~\ref{lem:3isat:ivp-has-sol}, we know that the solution $y$ to~\eqref{eq:3isat:ivp-dx} can be extended to a solution of the IVP defined on the full
  interval $\lt]0,1\rt]$.  Moreover, $\frac{dx}{dt} < 0$ whenever $13y(x) < 12x$, so the derivative of $x$ is strictly negative provided that $x \ge x_0$.  This
  implies that the solutions $x$, $y_2$, $y_3$ to~\eqref{eq:outer:ivp-dt} can be extended to the interval $[0,\tau]$, where $\tau$ is the unique number satisfying
  $x(\tau) = x_0$; in particular we have $\sigma < \tau$.
  
  This completes the proof of the lemma.
\end{proof}

We are now ready to prove Theorem~\ref{thm:algo-works-whp}.

\begin{proof}[Proof of Theorem~\ref{thm:algo-works-whp}]
  Lemma~\ref{lem:outer:params-stickto-IVP} gives the behavior of the parameters $X(t)$, $Y_2(t)$, and $Y_3(t)$ up to an error with high
  probability for all $t=0,\dots,\tau n$.  We need to check that
  \begin{enumerate}[(a)]
  \item the algorithm terminates before $t$ grows beyond $\tau n$,
  \item in this region of $t$, whp, the algorithm does not produce an empty clause.
  \end{enumerate}

  For~(a), we solve the IVP~\eqref{eq:3isat:ivp-dx} numerically for $c=2.3$.  The solution is drawn in Fig.~\ref{fig:outer:sol-c2.3}.  The figure also shows the
  line $13y=12x$.  For this value of $c$, we see that there is an $\eps>0$ such that the solution $y(x)$ to the IVP~\eqref{eq:3isat:ivp-dx} satisfies $13y(x) <
  12(1+2\eps)x$ for all $x > 2\eps$; w.l.o.g., we may assume that $\eps < \nfrac{1}{9}$.  Consequently, the $x_0$ from Lemma~\ref{lem:outer:params-stickto-IVP}
  equals $3\eps$.
  Algorithm~\ref{algo:3isat:outer} terminates as soon as $Y_2(t)+Y_3(t) \le c' X(t)$.  Thus, by Lemma~\ref{lem:outer:params-stickto-IVP}, we have an $s < \tau$ such
  that $x(s) = \nfrac13 > x_0$, and that, if we let $c' := \frac{50}{39}$, whp, for this $t := \lceil sn \rceil$
  \begin{multline*}
    Y_2(t) + Y_3(t)
    = ny(\nfrac13) + nc(\nfrac13)^3 +o(n)
    \\
    \leq n \left((1-2\eps)\tfrac{12}{13}\cdot\tfrac13 +
      \tfrac{c}{27}\right) + o(n)
    \le \tfrac{49}{39}\cdot \tfrac13 n
    \le c' X(t) -o(n),
  \end{multline*}
  if $n$ is large enough.
  Thus, the algorithm terminates before the parameters $X(\cdot)$, $Y_2(\cdot)$, $Y_3(\cdot)$ fail to be $2\eps$-good.

  It follows that Lemma~\ref{lem:inner:no-failure} gives a failure probability of~$o(\nfrac{1}{n})$ per iteration, so that the total failure probability is~$o(1)$.
  This proves~(b) and completes the proof of Theorem~\ref{thm:algo-works-whp}.
\end{proof}

\begin{figure}[htp]
  \centering
  \includegraphics[scale=.3333333333333333333333]{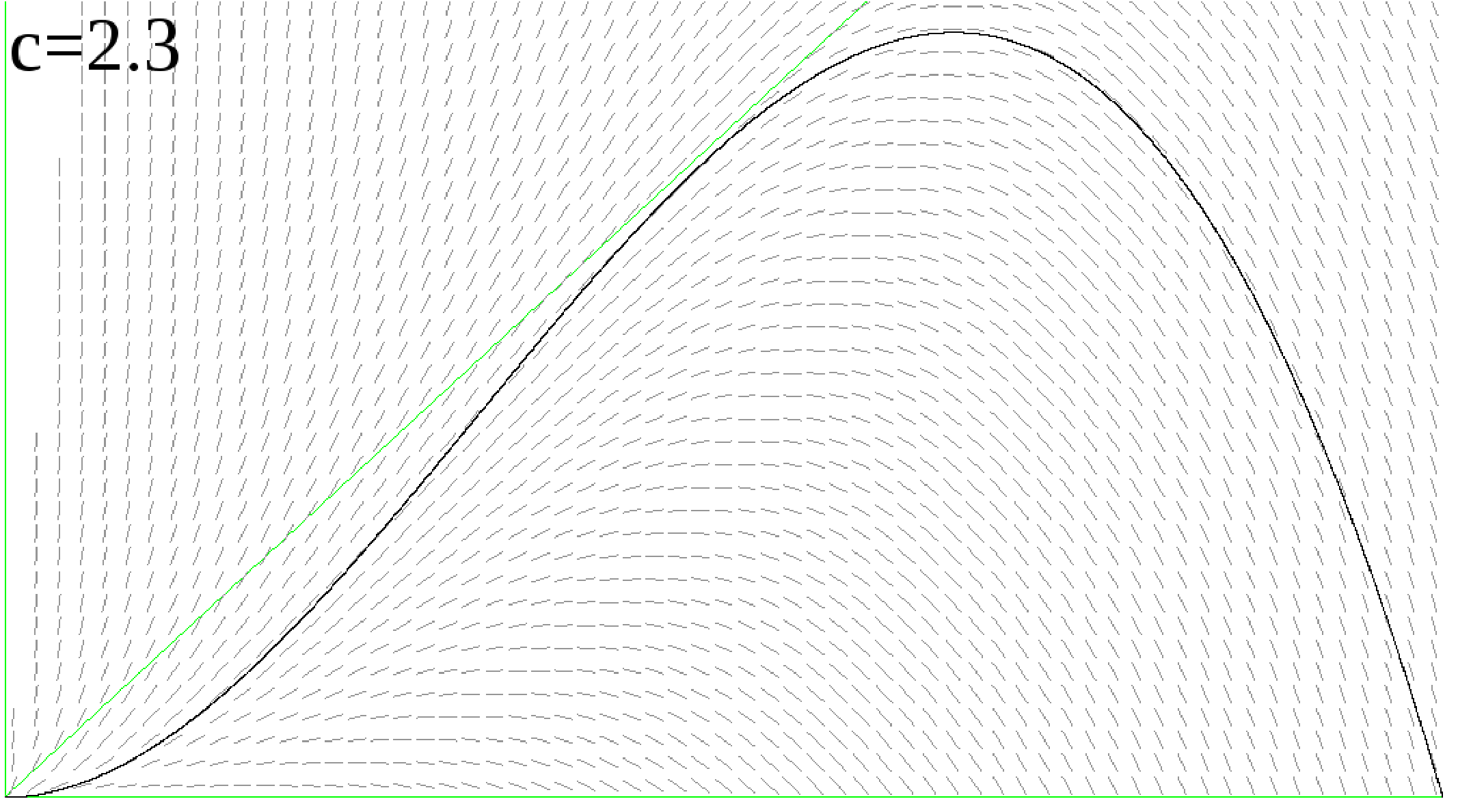}
  \caption{Solution of IVP~\eqref{eq:3isat:ivp-dx} with bounding curves}\label{fig:outer:sol-c2.3}
\end{figure}



\section{Conclusions and outlook}

The presented algorithm and its analysis provide a first systematic approach to random iSAT formulas.  In the course of the paper, analytical methods for dealing
with intervals in CNF formulas have been established, in particular as intervals imply correlation between the variables when choosing a value.  These tools will be
useful in the study of random algorithms for iSAT as well as in approximating a threshold in random 3-iSAT.

We have given an algorithm for $k$-iSAT, for $k=3$, which succeeds with high probability on instances for which $\nfrac{m}{n} \le \mbox{2.3}$.
It is conceptually easy to extend the algorithm and the analysis to general $k$ up to the point where the initial value problem has to be solved. For $k$-iSAT there
are $k-2$ ordinary differential equations to be numerically solved after the transformation in Lemma~\ref{lem:outer:params-stickto-IVP}, which makes it improbable
that a general formula for the maximal ratio can be derived.
Solving the system for small values of $k$, we obtain the results shown in Table~\ref{tab:other-k-values} (we always rounded down generously).

\begin{table}[htp]
  \centering
  \begin{tabular}{lrrrrrr}
    \toprule
    $k$                 & 3    & 4    & 5    & 6    & 7    & 8    \\
    max.~$\nfrac{m}{n}$  & 2.3 & 3.75 & 6.25 & 10.5 & 18.5 & 32.5 \\
    \bottomrule\\
  \end{tabular}
  \caption{Performance for different values of $k$}
  \label{tab:other-k-values}
\end{table}

It is possible to show that, whp, our algorithm fails to produce a satisfying interpretation if $\nfrac{m}{n} = c$ where $c$ is a constant for which the solution to
the IVP~\eqref{eq:3isat:ivp-dx} crosses the line $13y = 12 x$ (the green line in Fig~\ref{fig:outer:sol-c2.3}), e.g., $c=2.4$.  This is so because then the inner
loop runs for $\Omega(n)$ steps, and thus, whp, the algorithm \hurzes.  (However, such a result appears futile, given the very limited repair routine which we refer
to.)

\smallskip%
\paragraph{\bf Some further questions will be of interest.}
Firstly, the proposed algorithm can be improved in an obvious way: Whenever a variable is set, choose a value which is satisfies the maximum number of literals
containing the variable.  This, however, requires that the following question be answered.  Let $\lambda$ be a nonnegative real number.  Suppose that
$I_0,I_1,I_2,\dots$ are random invervals drawn independently uar from the sub-intervals of $[0,1]$, and $N$ is a Poisson random variable with mean~$\lambda$,
independent from the $I_j$.  What is the expectation $\xi(\lambda)$ of the following random variable?
\begin{equation*}
  \max\Bigl\{ \abs{K} \Bigm| K \subset \{1,\dots,N\}\text{, }  I_0\cap\bigcap_{j\in K} I_j \ne \emptyset \Bigr\}?
\end{equation*}

Secondly, a bound for the ratio above which random 3-iSAT formulas are wpp/whp not satisfiable might be interesting and worthwhile to be considered.

Thirdly, there might be a threshold for random 2-iSAT as for classical 2-SAT~\cite{ChvatalReed92,Goerdt94}.
It would be desirable to have computational indication of the existence of such a treshold.  (Such evidence
exists for the ``regular'' iSAT variant, where the endpoints of the constraint intervals must intersect
$\{0,1\}$.)  We conjecture that there is a threshold at $c = \nfrac32$ (the value from
Proposition~\ref{prop:satisf-of-2sat}).

More generally, it may be of interest whether the results of Friedgut (and Bourgain)~\cite{Friedgut99} (see
also \cite{Molloy02,Molloy03,CreignouDaude03,CreignouDaude04,CreignouDaude09}) can be applied to random
iSAT formulas to prove that a threshold (function) exists for $k$-iSAT for $k\ge 3$.

Fourthly, possibly, a stronger bound for 3-iSAT could be derived by adapting the algorithm of~\cite{KaporisKirousisLalas06} to the interval case.  This would pose
two problems: First we are interested in a whp result, which is not offered by the algorithm in~\cite{KaporisKirousisLalas06}, so a backtracking routine would have
to be devised; Secondly, the rule for the value assignment significantly complicates the computations for random intervals.  In their algorithm a randomly chosen
variable is assigned the value such that most clauses, in which it occurs, are satisfied, i.e., a variable is assigned a $1$ if it mostly occurs not negated, and $0$
otherwise.  For intervals this translates to assigning a value to a variable that is contained in the non-empty intersection of a maximal number of associated
intervals.  But the analysis of the probability of this maximal number turns out to be demanding for general intervals.

Finally, we would like to draw attention to the fact that several papers have raised questions concerning
the existence and location of a threshold for random regular
3-iSAT~\cite{BeckertHaehnleManya99,BejarManya99phase,BejarManyaCabiscolFernandezGomes07,ManyaBejarEscaladaimaz98}.

\bigskip\noindent%
{\bf We would like to close by thanking} %
the anonymous referees for their very valuable comments!



\providecommand{\bysame}{\leavevmode\hbox to3em{\hrulefill}\thinspace}
\providecommand{\MR}{\relax\ifhmode\unskip\space\fi MR }
\providecommand{\MRhref}[2]{%
  \href{http://www.ams.org/mathscinet-getitem?mr=#1}{#2}
}
\providecommand{\href}[2]{#2}

%


\appendix

\section{Deferred proofs}

\subsection{Computations for Lemma~\ref{lem:rint:moments-of-P}}\label{apx:rint:moments-of-P}

For (a), we compute
\begin{multline*}
  1-\Exp P 
  = \Exp(2X(1-X))
  = \int 2t(1-t) \,dF(t)
  \\
  = \int_{\lt[0,\nfrac12\rt[} 2t(1-t) \partial_t F(t) \,dt
  + 2t(1-t)\Big|_{t=\nfrac12} \cdot \tfrac12
  + \int_{\lt]\nfrac12,1\rt]} 2t(1-t) \partial_t F(t) \,dt
  \\
  = \int_0^{\nfrac12} 2t(1-t) 2 t \,dt
  + \frac14
  +  \int_{\nfrac12}^1 2t(1-t) 2(1-t) \,dt
  \\
  = \frac{5}{48}
  + \frac14
  + \frac{5}{48}
  = \frac{11}{24}.
\end{multline*}
For (b), we compute
\begin{multline*}
  \Exp(X^2(1-X)^2)
  = \int t^2(1-t)^2 \,dF(t)
  \\
  = t^2(1-t)^2\Big|_{t=\nfrac12} \cdot \tfrac12
  + \int_{\lt[0,\nfrac12\rt[} t^2(1-t)^2 \partial_t F(t) \,dt
  + \int_{\lt]\nfrac12,1\rt]} t^2(1-t)^2 \partial_t F(t) \,dt
  \\
  = \tfrac{1}{2^5}
  + \int_{\lt[0,\nfrac12\rt[} t^2(1-t)^2 2t \,dt
  + \int_{\lt]\nfrac12,1\rt]} t^2(1-t)^2 2(1-t) \,dt
  \\
  = \tfrac{1}{2^5}
  + 4 \int_0^{\nfrac12} t^3(1-t)^2 \,dt
  \\
  = \tfrac{1}{2^5}
  + 4 \lt( \tfrac14 t^4(1-t)^2\Big|_{t=0}^{\nfrac12} + \tfrac{1}{10} t^5(1-t)\Big|_{t=0}^{\nfrac12} + \tfrac{1}{60} t^6\Big|_{t=0}^{\nfrac12} \rt)
  \\
  = \tfrac{1}{2^5}
  + 4 ( \tfrac14 \tfrac{1}{2^6} + \tfrac{1}{10}\tfrac{1}{2^6} + \tfrac{1}{60}\tfrac{1}{2^6} )
  = \tfrac{1}{2^5}
  +  \tfrac{1}{2^6}( 1 + \tfrac{2}{5} + \tfrac{1}{15} )
  = \tfrac{1}{2^5} + \tfrac{22}{15\cdot 2^6}
  = \tfrac{1}{2^5} + \tfrac{11}{15\cdot 2^5}
  = \tfrac{15 + 11}{15\cdot 2^5}
  = \tfrac{13}{15\cdot 2^4}.
\end{multline*}
Hence, using (a), we obtain
\begin{equation*}
  \Exp(P^2)
  = 1 - 2(1-\Exp P) + 4\Exp X^2(1-X)^2
  = 1 - \tfrac{11}{12} + \tfrac{13}{60}
  = \tfrac{18}{60}
  = \tfrac{3}{10}.
\end{equation*}

\subsection{Proof of Lemma~\ref{lem:Q:qlength-properties}}\label{apx:Q:qlength-properties}

The proof is taken almost word for word from Grimmett \& Stirzaker~\cite{GrimmettStirzaker01}, Theorem~11.3.17, with some changes due to the discrete arrival- and
servicing points.
  
We say that the \textit{sons} of a customer Paul are those customers arriving in the time interval in which Paul is serviced.  Paul's \textit{family} consists of
himself and all of his descendants.

Fix a time interval $j$ in which the queue is not empty and denote by $X$ the size of the family of the customer served at that time interval.  We have the relation
\begin{equation*}
  X = 1 + \sum_{i=1}^{B(j+1)} X_i,
\end{equation*}
where $X_i$ denotes the family size of the $i$'th customer arriving in the time interval $j$.

The important observation now is that the family sizes are iid because the $B(j)$ are iid, and that the $X_i$ are independent of $B(j+1)$.
Consequently, for the common probability generating function $y$ of $X$ and the $X_i$, we have
\begin{equation}\label{eq:Q:textbook:ast}\tag{$*$}
  y(x) = x \ \gfun{B}(y(x)).
\end{equation}
The length of the first busy period coincides with the sum of the family sizes of the~$a$ customers
arriving in the first time interval.  Thus, we obtain
\begin{equation}\label{eq:Q:textbook:astast}\tag{$**$}
  h(x) = y(x)^a.
\end{equation}
Solving~\eqref{eq:Q:textbook:ast} for $x$ and inserting into~\eqref{eq:Q:textbook:astast}, we obtain
\begin{equation}\label{eq:Q:textbook:3ast}\tag{$*$$*$$*$}
  h\bigl( \tfrac{y(x)}{\gfun{B}(y(x))} \bigr) = y(x)^a.
\end{equation}
If $y(0) = 0$, then $B=0$, and thus $h(y) = y^a$, which coincides with equation~\eqref{eq:Q:probGenF}.
Otherwise, by~\eqref{eq:Q:textbook:3ast}, equation~\eqref{eq:Q:probGenF} holds for all $y$ in the interval $[y(0),y(1)]$, and thus for all $y$ for which the power
series on both sides of the equality sign converge.
  
We derive the statement about the mean length of the first busy period by differentiating~\eqref{eq:Q:probGenF}, and possibly invoking Abel's Theorem to evaluate
the power series at the point $1$.
  
Finally, the statement about the tail probability follows directly from the standard exponential moment argument: If $y\ge \gfun{B}(y) > 0$, then, with $x :=
\nfrac{y}{\gfun{B}(y)} \ge 1$, we have
\begin{equation*}
  \Prb[ Z \ge \alpha ]
  = \Prb[ x^Z \ge x^\alpha ]
  \le \frac{\Exp x^Z}{x^\alpha}
  = \frac{h(x)}{x^\alpha}
  = \frac{y^a}{(\nfrac{y}{\gfun{B}(y)})^\alpha}
  = \frac{\gfun{B}(y)^\alpha}{y^{\alpha-a}},
\end{equation*}
as claimed.

\subsection{Computations for Lemma~\ref{lem:Q:momgen-estimate}}\label{apx:Q:iPoi:busy-prd-tail}

Computations regarding equation~\eqref{eq:innerloop:exponent:osudfbo832hrlsdin}:
\begin{gather*}
  \alpha r + \tfrac{12^3}{13^2\cdot 5} \alpha r^2 u - (\alpha-a)\frac{1}{u+1} = 0\\
  \alpha r(u+1) + \tfrac{12^3}{13^2\cdot 5} \alpha r^2  u(u+1) - (\alpha-a) = 0\\
  \Bigl( \tfrac{12^3}{13^2\cdot 5} \alpha r^2  \Bigr)u^2 + \Bigl(  \alpha r  + \tfrac{12^3}{13^2\cdot 5} \alpha r^2  \Bigr)u - \bigl( (1-r)\alpha - a \bigr) = 0\\
  u = - \frac{\Bigl(  \alpha r  + \tfrac{12^3}{13^2\cdot 5} \alpha r^2  \Bigr)
    \pm\sqrt{ \Bigl(  \alpha r  + \tfrac{12^3}{13^2\cdot 5} \alpha r^2  \Bigr)^2
      + 4\bigl( (1-r)\alpha - a \bigr)\Bigl( \tfrac{12^3}{13^2\cdot 5} \alpha r^2  \Bigr)
    }
  }{2\cdot \Bigl( \tfrac{12^3}{13^2\cdot 5} \alpha r^2  \Bigr) }\\
\end{gather*}
We need to be close to $0$, so we take the ``$\pm$'' $=$ ``$+$'':
\begin{gather*}
  u_r
  := \frac{- \Bigl( \alpha r  + \tfrac{12^3}{13^2\cdot 5} \alpha r^2  \Bigr)
    +\sqrt{ \Bigl(  \alpha r + \tfrac{12^3}{13^2\cdot 5} \alpha r^2  \Bigr)^2
      + 4\bigl( (1-r)\alpha - a \bigr) \tfrac{12^3}{13^2\cdot 5} \alpha r^2
    }
  }{2\cdot \tfrac{12^3}{13^2\cdot 5} \alpha r^2 }\\
  =
  \frac{- \Bigl( 1  + \tfrac{12^3}{13^2\cdot 5} r  \Bigr)
    +\sqrt{ \Bigl(1 + \tfrac{12^3}{13^2\cdot 5} r  \Bigr)^2
      + 4\bigl(1 - r - \nfrac{a}{\alpha} \bigr) \tfrac{12^3}{13^2\cdot 5}
    }
  }{2\cdot \tfrac{12^3}{13^2\cdot 5} r }\\
  =
  \frac{- \Bigl(  1  + \tfrac{12^3}{13^2\cdot 5} r  \Bigr)
    +\sqrt{ \Bigl(  1  + \tfrac{12^3}{13^2\cdot 5} r  \Bigr)^2
      -4r \tfrac{12^3}{13^2\cdot 5}
      + 4(1-\nfrac{a}{\alpha})\tfrac{12^3}{13^2\cdot 5}
    }
  }{2\cdot \tfrac{12^3}{13^2\cdot 5} r }\\
  =
  \frac{- \Bigl(1  + \tfrac{12^3}{13^2\cdot 5} r  \Bigr)
    +\sqrt{ \Bigl(  1  - \tfrac{12^3}{13^2\cdot 5} r  \Bigr)^2
      + 4(1-\nfrac{a}{\alpha})\tfrac{12^3}{13^2\cdot 5}
    }
  }{2\cdot \tfrac{12^3}{13^2\cdot 5} r }\\
  =
  \frac{- \Bigl(  1  + \tfrac{12^3}{5\cdot 13^2} r  \Bigr)
    +\sqrt{ \Bigl(  1  - \tfrac{12^3}{5\cdot 13^2} r  \Bigr)^2
      + \tfrac{4\cdot 12^3}{5\cdot 13^2}     -\tfrac{4\cdot 12^3}{5\cdot 13^2}\cdot\frac{a}{\alpha}
    }
  }{\tfrac{2\cdot 12^3}{5\cdot 13^2} r }\\
  =
  \frac{- \Bigl( 1  + \tfrac{12^3}{5\cdot 13^2} r  \Bigr)
    +\sqrt{ \Bigl(1  - \tfrac{12^3}{5\cdot 13^2} r  \Bigr)^2
      + \tfrac{4\cdot 12^3}{5\cdot 13^2}
    }
  }{\tfrac{2\cdot 12^3}{5\cdot 13^2} r }
  - O(\nfrac{a}{\alpha}),
\end{gather*}
with an absolute constant in the $O(\cdot)$, because $a\le\alpha$ and $\nfrac12 \le r \le 1$.

Computation regarding equation~\eqref{eq:Q:iPoi:pre-tail}:
\begin{gather*}
  \frac{\fbox{($*$)}(u_r)}{\alpha}
  = \frac{\lt. \alpha r u + \tfrac{12^2\cdot 3 \cdot 2}{13^2\cdot 5} \alpha r^2 u^2 - (\alpha-a)\log(u+1) \rt|_{u:=u_r}}{\alpha}
  \\
  = r u_r  + \tfrac{6\cdot 12^2}{5\cdot 13^2}  r^2 u_r^2 - (1-\nfrac{a}{\alpha})\log(u_r+1)
  \\
  = ru_r + \tfrac{6\cdot 12^2}{5\cdot 13^2}  r^2 u_r^2 - \log(u_r+1)   + O(\nfrac{a}{\alpha}),
\end{gather*}
with an absolute constant in the $O(\cdot)$, because $u_r+1 \le 2$.

\subsection{Proof of Lemma~\ref{lem:Q:final}}\label{apx:Q:final}

  Suppose that the $B(j)$ are represented as a sum as in~\eqref{eq:Q:dBin-coupling} above, and define
  \begin{equation*}
    B^\pm(j) := \sum_{j=1}^{m^\pm} \Ind\Bigl[   U(j,i) \le \tfrac{P(j)}{n-(\pm z)}   \Bigr].
  \end{equation*}
  Then the $B^+(j)$, $j=1,2,3,\dots$, are iid, so that Lemma~\ref{lem:Q:binomial:mean-and-tail} is applicable.  The same is true for the $B^-(j)$, $j=1,2,3,\dots$.
  We clearly have, with probability $1-O(n^{-r})$,
  \begin{equation*}
    B^-(j) \le B(j) \le B^+(j) \qquad\text{ for all $j=1,\dots,z$}.
  \end{equation*}
  Defining two queues $Q^\pm(j)$ based on the $B^\pm(j)$ and respective lengths of first busy periods $Z^\pm$, we obtain, with probability $1-O(n^{-r})$
  \begin{equation}\label{eq:z-_le_z_le_z+}\tag{$*$}
    Z^- \le Z \le Z^+,
  \end{equation}
  where we have also used that $Z^\pm \le z$ with probability $1-O(n^{-r})$ (Lemma~\ref{lem:Q:binomial:mean-and-tail}).
  
  Denote by $E$ the event that~\eqref{eq:z-_le_z_le_z+} holds.  If~\eqref{eq:z-_le_z_le_z+} does not hold, we still have $Z = O(n)$
  by~\eqref{eq:Q:ub-total-no-customers}, so that we obtain
  \begin{equation*}
    \Exp Z
    = \Exp(Z\mid E)\Prb(E) + \Exp(Z\mid \widebar E)\Prb(\widebar E) 
    \le \Exp(Z^+\mid E)\Prb(E) + O(n^{1-r})
    \le \Exp(Z^+) + O(n^{1-r}).
  \end{equation*}
  For the lower bound, we similarly have
  \begin{equation*}
    \Exp Z
    \ge \Exp( \Ind(E) Z^- )
    = \Exp(Z^-) - \Exp( \Ind(\widebar E) Z^- )
  \end{equation*}
  Clearly, $\Exp( \Ind(\widebar E) Z^- ) \le z\Prb(\widebar E) + \Exp(\Ind(\widebar E) Z^- \Ind[Z^- > z]) = z\Prb(\widebar E) + mO(n^{-r}) = O(n^{1-r})$
  Thus we conclude that $\Exp Z \ge \Exp Z^- - O(n^{1-r})$.
  
  For the tail estimate, we use $Z^+$:
  \begin{multline*}
    \Prb[ Z \ge \alpha ]
    \le \Prb[ Z \ge \alpha \aand Z \le Z^+] + \Prb[ Z \ge \alpha \aand Z > Z^+]
    \\
    \le \Prb[ Z^+ \ge \alpha ] + \Prb[Z > Z^+]
    \le e^{-\delta\alpha} + O(n^{-r})
  \end{multline*}
  by Lemma~\ref{lem:Q:binomial:mean-and-tail}.

\subsection{Computations for Lemma~\ref{lem:3isat:ivp-has-sol}}\label{apdx:3isat:ivp-has-sol}

\subsubsection*{\it For the proof of Claim~(\ref{enum:3isat:ivp-has-sol:left}).}

Let $g(x,y) := \frac{-18cx^4 + 2y(12x-y)}{x(12x - y)}$, the right hand side of the ODE~\eqref{eq:3isat:ivp-dx:ode}.  As mentioned in the proof of the lemma, we
show $g(x,y(x)) > 6 = \frac{dy}{dx}$, for $0<x\le\nfrac45$.
We compute
\begin{multline*}
  g(x,y(x))
  = \frac{-18cx^4 + 2\cdot 6x (12x-6 x)}{x(12x - 6x)}
  = \frac{-18cx^2 + 2\cdot 6 (12-6)}{(12 - 6)}
  = \frac{-18cx^2 + 72}{6}
  \\
  = -3cx^2 + 12
  \gecmt{c\le 3} 
  -9x^2 + 12
  \ge -9(\nfrac45)^2 + 12
  = 12-\frac{9\cdot 16}{25}
  = \frac{25\cdot 12 - 9\cdot 16}{25}
  \\
  = \frac{12(25 - 3\cdot 4)}{25}
  = \frac{12\cdot 13}{25}  > 6.
\end{multline*}

\subsubsection*{\it For the proof of Claim~(\ref{enum:3isat:ivp-has-sol:right}).}

Let $g(x,y)$ as above.  As mentioned in the proof of the lemma, we show $g(x,y(x)) > -5 = \frac{dy}{dx}$, for $\nfrac45 \le x \le 1$.
To show that 
\begin{equation*}
  g(x,y(x)) = \frac{-18cx^4 + 2\cdot 5(1-x)(12x-5(1-x))}{x(12x - 5(1-x))} > -5,
\end{equation*}
we compute
\begin{multline*}
  -18cx^4 + 2\cdot 5(1-x)(12x-5(1-x)) + 5x(12x - 5(1-x))
  \\
  = -18cx^4 + 10(1-x)(17x-5) + 5x(17x - 5)
  = -18cx^4 + (10-5x)(17x-5)
  \\
  = -18cx^4 -85x^2 +195x -50
  \gecmt{c\le3}
  -54x^4 -85x^2 +195x -50.
\end{multline*}
The derivative $-216 x^3 -170x +195$ of the last polynomial is strictly decreasing, and evalutating it at $\nfrac45$ gives
$-216 (\nfrac45)^3 -170\cdot \nfrac45 +195 \approx -51.592 < 0$.  Thus, it suffices to check the inequality
$-54x^4 -85x^2 +195x -50 > 0$ for $x=1$: $-54 -85 +195 -50 = 6 > 0$.


\end{document}